%% file: Pan_manuscript_clean.tex
\documentclass[hidelinks,onefignum,onetabnum]{siamart220329}

\usepackage[normalem]{ulem}

\input{ex_shared}

\ifpdf
\hypersetup{
  pdftitle={A Two-Sided Sketching Algorithm for Low-rank Tensor Train Approximation},
  pdfauthor={Gaohang Yu, Yihao Pan, Ailun Jian, Xiaohao Cai}
}
\fi

\begin{document}

\maketitle

\begin{abstract}
Tensor train (TT) decomposition is a powerful method to acquire low-rank tensors. However, the computational process is frequently obstructed by the large-scale matrix singular value decomposition (SVD). The sketching algorithm serves as an efficient data compression technique that can quickly derive low-rank matrix approximations. In this paper, we propose a randomized algorithm to obtain the TT approximation of tensors using a one-pass sketching algorithm and subspace iteration, and offer thorough error-bound and robustness analysis. Numerical experiments on synthetic and real-world datasets demonstrate the effectiveness and efficiency of the proposed algorithm.
\end{abstract}



\section{Introduction}
In real-world scenarios, large high-dimensional datasets can often be naturally expressed as tensors, frequently showing a low-rank structure that can be described with significantly fewer parameters. As a result, low-rank tensor approximations have been widely used in various areas, including machine learning, graph analysis, and signal processing, among others \cite{kolda2009}\cite{comon2014}\cite{grasedyck2013}.
Traditional tensor decomposition techniques encompass canonical (CP) decomposition \cite{comon2004}, Tucker decomposition \cite{de2000}, tensor singular value decomposition (t-SVD)\cite{kilmer2011}, tensor train (TT) decomposition \cite{oseledets2011}, tensor ring (TR) decomposition \cite{zhao2016}, and others. CP decomposition can effectively break down a tensor into the sum of multiple rank-1 tensors. However, determining the exact CP rank is challenging. Although Tucker decomposition provides enhanced stability over CP decomposition, it is plagued by the curse of dimensionality. In contrast, TT decomposition overcomes the curse of dimensionality and is more dependable. This paper primarily focuses on TT decomposition, which is gaining popularity due to its robustness and efficiency.
        
TT decomposition involves breaking down large-scale tensors into several third-order tensors. One straightforward method is TT-SVD \cite{oseledets2011}, which utilizes auxiliary unfolding matrices for SVD. Nonetheless, conducting SVD on large-scale matrices is highly time-intensive. There are various methods to make the computation both efficient and accurate. One particular category is randomized methods \cite{yu2018,batselier2018,tropp2019,martinsson2020}, which have shown outstanding efficiency and reliability with large-scale data. Halko et al. \cite{halko2011} introduced a randomized SVD (R-SVD) technique that effectively computes low-rank approximations of large-scale matrices. When the original data is so vast that it can only be read once by the core memory, so one-pass algorithms become essential. Tropp et al. \cite{tropp2017} proposed a one-pass sketching algorithm for streaming data, which can create low-rank approximations of input matrices from random linear sketches of the matrix and also provide error bounds.
     
In recent years, many scholars have extended matrix random methods to tensors, yielding promising results \cite{sun2020,minster2020,chen2023,che2020,daas2022}. Based on R-SVD, Huber et al. \cite{huber2017} proposed a randomized TT decomposition that is dramatically faster than TT-SVD. Che et al. \cite{che2019} provided an adaptive random algorithm to compute a low multilinear rank approximation of tensors with unknown multilinear rank. Shi et al. \cite{shi2023} proposed parallelizable algorithms to compute the decomposition of the TT with theoretical guarantees of precision. Ma et al. \cite{ma2023} recovered a low-rank matrix and low-tubal-rank tensor from a noisy sketch with rigorous theoretical guarantees.
	
Power iteration can improve the accuracy of approximations by reducing tail energy, making it more accurate than simple random projection algorithms. Vogel et al. \cite{vogel1994} introduced the subspace iteration method and considered the efficient iterative computation of partial SVDs. Gu \cite{gu2015} presented a novel error analysis considering randomized algorithms within the subspace iteration framework. Yu et al. \cite{yu2023}  proposed a randomized algorithm for low-rank TT approximation based on randomized block Krylov subspace iteration. Dong et al. \cite{dong2023} combined subspace power iteration with two-sided sketching algorithms and applied it to compute the Tucker approximation of tensors. Zhang et al. \cite{zhang2018} proposed a method that extended a well-known randomized matrix SVD method to the t-SVD, and presented an improved analysis of the randomized simultaneous iteration for matrices.
	 
In this paper, we propose a randomized algorithm named TT-subSKETCH for TT approximations of tensors. {This algorithm utilizes two-sided sketching for low-rank matrix approximation to enhance its efficiency.} By integrating the sketching method with the power iteration approach, we can obtain higher quality TT approximations and deliver thorough error analysis. Additionally, we address low-rank approximations in noisy conditions and expand upon the findings of \cite{ma2023} in the context of TT decomposition. 
	 
The rest of this paper is organized as follows. We recall some background and related methods on TT decomposition in Section \ref{section:2}. In Section \ref{section:3}, we extend the error bounds of the subspace power iteration with two-sided sketching to the \( \ell_2 \)-norm and provide our proposed algorithm TT-subSKETCH with error analysis; moreover, we also provide error bounds for noisy sketch of tensor. Thorough numerical experiments on synthetic and real-world data tensor are given in Section \ref{section:4},  verifying the effectiveness and robustness of the proposed algorithm. Finally, we conclude in Section \ref{section:5}.

\section{Background}
\label{section:2}

\subsection{Notations and basic operations}
In this paper, matrices are denoted by capital letters like $A$, tensors by Euler script letters like $\mathcal{A}$, and $\mathbb{R}$ represents the real number space. For a $d$-th order tensor $ \mathcal{A}\in\mathbb{R}^{n_1\times n_2\times\cdots\times n_d} $, its $ (i_1,i_2,\cdots ,i_d) $-th element is represented by $\mathcal{A}(i_1,i_2,\cdots ,i_d) $. For matrix $A$ with orthogonal columns, the notation $A_\perp$ represents the orthogonal complement of $A$, which means that the column vectors of $A$ and $A_\perp$ form a complete orthogonal basis. Define the Frobenius norm of $ \mathcal{A} $ as
	    \begin{equation}
	    	 \lVert \mathcal{A} \rVert_F := \sqrt{\langle \mathcal{A},\mathcal{A} \rangle} =  \sqrt{\Sigma_{i_1,i_2,\cdots ,i_d} \mathcal{A}(i_1,i_2,\cdots ,i_d)^2 } .
    	\end{equation}

The mode-$ \alpha $ product of tensor $ \mathcal{A}\in\mathbb{R}^{n_1\times n_2\times\cdots\times n_d} $ by matrix $ B\in\mathbb{R}^{M\times n_\alpha} $ is designated as $ \mathcal{A} \times_\alpha B = \mathcal{C} \in \mathbb{R}^{n_1\times\cdots\times n_{\alpha-1}\times M\times n_{\alpha+1}\times\cdots\times n_d} $, with entries
        \begin{equation}
		\mathcal{C}(i_1,\cdots,i_{\alpha-1},m,i_{\alpha+1},\cdots,i_d) = \! \sum_{i_j=1}^{n_\alpha}\mathcal{A}(i_1,\cdots,i_{\alpha-1},i_j,i_{\alpha+1},\cdots,i_d)B(m,i_j).
	\end{equation} 
The tensor-tensor product of two tensors $ \mathcal{A}\in\mathbb{R}^{n_1\times \cdots\times n_d} $ and $ \mathcal{B}\in\mathbb{R}^{m_1\times \cdots\times m_e} $ with equal the size of the $\alpha$-mode of the tensors $ n_\alpha=m_\beta $ produces a $ (d+e-2) $-th order tensor $ \mathcal{C} $, i.e.,
	 \begin{equation}	 	\mathcal{C}=\mathcal{A}\times^\beta_\alpha\mathcal{B}.
	 \end{equation}
where $ \mathcal{C}\in\mathbb{R}^{n_1\times \cdots\times n_{\alpha-1}\times n_{\alpha+1}\times\cdots\times n_d\times m_1\times\cdots\times m_{\beta-1}\times m_{\beta+1}\times\cdots\times m_e}. $

\begin{table} [h]
	\centering
	\begin{tabular}{cc}
		\hline
		Symbol &  Description\\
		\hline
		$\mathbb{R}$&real number space\\
		$p$ &   oversampling parameter\\ 
		$q$ & the number of power iterations\\
		${A}$	& matrix \\
		$\mathcal{A}$	& tensor  \\
		${\sigma_{i}}(A)$ & the ${(i)}_{th}$ singular value of $A$\\
		$ A^{T}$ & the transpose of $ A$ \\
		$A^\dagger$ & the pseudoinverse of $ A$\\
		$\hat{A}$ & the approximation tensor of $A$ \\
        $A_\perp$ & the orthogonal complement of $A$ \\
		$\| \cdot \|_2$ & Spectral norm\\
		$\| \cdot \|_F$ & Frobenius norm\\
		\hline
	\end{tabular}
	\caption{Symbol description} \label{Tab:symbol}
\end{table}

\begin{Def}
		\label{Def1}
		(Mode-$(1,2,\cdots,k)$ regular matricization \cite{kolda2009}). \\The mode-$ (1,2,\cdots,k) $ regular matricization of a tensor $\mathcal{A}$ is denoted by $A_{[k]}\in\mathbb{R}^{\prod_{n=1}^{k}I_n\times\prod_{n=k+1}^N I_n}$. It can be implemented by calling the reshape function in Matlab:
        \[
        A_{[k]}=\mbox{reshape}(\mathcal{A},\prod_{n=1}^{k}I_n,\prod_{n=k+1}^N I_n),
        \]
    \noindent where the inverse operator of reshape is denoted by ``unreshape'',\\ i.e., $\mathcal{A}=\text{unreshape}(A_{[k]})$.
\end{Def}

\begin{Def}
    	\label{Def2}
    	(Tensor train format \cite{oseledets2011}). Let $ \mathcal{A}\in\mathbb{R}^{n_1\times n_2\times\cdots\times n_d} $ be a tensor of order d. A factorization
    	\begin{equation}		\mathcal{A}=\mathcal{G}_1\times_3^1\mathcal{G}_2\times_3^1\cdots\times_3^1\mathcal{G}_d
    	\end{equation}
    	of $\mathcal{A} $, into core tensors $\mathcal{G}_i\in\mathbb{R}^{r_{i-1}\times n_i\times r_i} (r_0=r_d=1) $, is called a TT decomposition of $\mathcal{A} $. The array of the dimensions $ \mathbf{r}=(r_1,\cdots,r_{d-1}) $ is the TT-rank of $ \mathcal{A} $ defined as	
    \begin{equation}
    	\mbox{rank}_{\mbox{TT}}(\mathcal{A})=(r_1,r_2,\ldots,r_{d-1})=\left(\mbox{rank}(A_{[1]}),\mbox{rank}(A_{[2]}),\ldots,\mbox{rank}(A_{[d-1]})\right).
    \end{equation}
\end{Def}

\begin{Def}
		\label{Def3}
		(Tail Energy \cite{higham1989}). The $j$-th tail energy of matrix $ A $ is defined as
			\begin{equation}
					\tau^2_j(A)=\min\limits_{{\rm rank}(B)<j}\lVert A-B \rVert^2_F = \sum\limits_{i\ge j}\sigma^2_i(A),
			\end{equation}
	where $\sigma_i(A)$ is the $i$-th singular value of $A$.
\end{Def}

\subsection{TT-SVD}
{
The TT low-rank approximation model for solving the tensor $\mathcal{A}$ can be expressed as\cite{che2024}
\begin{equation}\label{3.1}
  \mathop{\min }\limits_{\mathcal{G}_1,\mathcal{G}_2,\cdots,\mathcal{G}_d}\Vert \mathcal{A}-\mathcal{G}_1\times_3^1\mathcal{G}_2\times_3^1\cdots\times_3^1\mathcal{G}_d\Vert_F,
\end{equation}
where $\mathcal{G}_i\in\mathbb{R}^{r_{i-1}\times n_i\times r_i} (r_0=r_d=1,i\in [1,d]) $
for $i=1,2,\cdots,d-1,$ the TT core $\mathcal{G}_i$ satisfy
\begin{center}
  $G_i^\top G_i=I_{r_i}, G_i=reshape(\mathcal{G}_i,[r_{i-1}n_i,r_i]).$
\end{center}
Assume ${\mathcal{T}_1,\mathcal{T}_2,\cdots,\mathcal{T}_d}$ form a
solution of (\ref{3.1}). Let $T_i=reshape(\mathcal{T}_i,[r_{i-1}n_i,r_i]) \\(i=1,2,\cdots,d-1)$. Define
\begin{center}
  $\begin{cases}
  &\mathcal{A}_0=\mathcal{A}, \\
  &\mathcal{A}_1=\mathcal{T}_1\times_1^1\mathcal{A}_0\in\mathbb{R}^{r_1\times n_2\times\cdots\times n_d}, \\
  &\mathcal{A}_i=\mathcal{T}_i\times_{1,2}^{1,2}\mathcal{A}_{i-1}\in\mathbb{R}^{r_i\times n_{i+1}\times\cdots\times n_d}, i=2,3,\cdots,d-1,\\
  &A^{(i)}=reshape(\mathcal{A}_i,[r_{i-1}n_i,n_{i+1}\cdots n_d]), i=1,2,\cdots,d-1.
\end{cases}$
\end{center}
We have
\begin{equation}\label{eq11}
  \Vert \mathcal{A}-\mathcal{T}_1\times_3^1\mathcal{T}_2\times_3^1\cdots\times_3^1\mathcal{T}_d\Vert_F\leq \sum_{i=1}^{d-1} \left\| A^{(i)} - T_i T_i^\top A^{(i)} \right\|_F.
\end{equation}
An approximate solution can be obtained by solving the following d-1 subproblem
for $i=1,2,\cdots,d-1$, when $r_i\leq min\{r_{i-1}n_i,n_{i+1},\cdots,n_d\}(r_0=1)$, the goal is to find an orthogonal matrix $T_i\in\mathbb{R}^{r_{i-1}n_i\times r_i}$ that satisfies
\begin{equation}
  T_i=\arg \min_{Q_i}\left\| A^{(i)}-Q_iQ_i^\top A^{(i)} \right\|_F,
\end{equation}
where $Q_i\in\mathbb{R} ^{r_{i-1}n_i\times r_i}$ is orthogonal.}

TT-SVD \cite{oseledets2011} stands as a classical algorithm for TT decomposition. Its essence lies in the intuition that at each step, an SVD of the auxiliary unfolding matrix is executed. The detailed TT-SVD algorithm is provided in Algorithm \ref{alg:TTSVD}. The {\tt numel} function in MATLAB returns the total number of elements in an array, while the {\tt reshape} function changes the size of an array without changing the data itself.

	\begin{algorithm}[H]
		\caption{TT-SVD \cite{oseledets2011}}
		\label{alg:TTSVD}
		\renewcommand{\algorithmicrequire}{\textbf{Input:}}
		\renewcommand{\algorithmicensure}{\textbf{Output:}}
		\begin{algorithmic}[1]
			\REQUIRE Tensor $ \mathcal{A}\in\mathbb{R}^{n_1\times n_2\times\cdots\times n_d} $, and prescribed accuracy $\varepsilon$ 
			\ENSURE Cores $ \mathcal{G}_1,\cdots,\mathcal{G}_d $ of the  TT-approximation $ \mathcal{B} $ 
			to $ \mathcal{A} $ in the TT-format satisfying $ \lVert \mathcal{A} -\mathcal{B}\rVert_F \leq \varepsilon \lVert \mathcal{A} \rVert_F $
\vspace{0.05in}
            
		    \STATE Compute truncation parameter $ \delta =\frac{\varepsilon}{\sqrt{d-1}}\lVert \mathcal{A} \rVert_F $
		    \STATE $ C=\mathcal{A}, r_0=1 $
			 
			\FOR{$k = 1$ to $d-1$}
			\STATE $ C :=$ {\tt reshape}$(C,[r_{k-1} n_k,\frac{{\tt numel}(\mathcal{A})}{r_{k-1} n_k}])$
			\STATE Compute $ \delta $-truncated SVD: $ C=USV+E, \lVert E \rVert_F \leq \delta , r_k = {\rm rank}_\delta (C) $
			
			\STATE New core: $ \mathcal{G}_k=$ {\tt reshape}$ (U,[r_{k-1},n_k,r_k]) $
			\STATE $ C=SV^\top $
			\ENDFOR
			\STATE $ \mathcal{G}_d=C$
			\RETURN $ \mathcal{G}_1 ,\mathcal{G}_2 , \cdots , \mathcal{G}_d $
		\end{algorithmic}
	\end{algorithm}
    
The computational complexity of the TT-SVD is controlled by the $(d-1)$ matrix SVDs. Let $ n:=\max(n_1,\cdots,n_d)$, $r:=\max(r_1,\cdots,r_d)$, and the cost scales be $ \mathcal{O}(n^{d+1}+\sum_{i=1}^{d-2}n^{d+1-i})$. Due to the structural nature of tensor data universally being disrupted after the initial SVD, sparse or structured tensors do not lead to significant computational efficiency improvements. The following theorem shows the error bound for TT-SVD.
    
\begin{theorem}\label{Thm1}
		(\cite{oseledets2011}, Theorem 2.2). Suppose that the k-matricization $ A_{[k]} $ of tensor $ \mathcal{A} $ satisfies 
		\begin{equation}
			 A_{[k]} =R_k+E_k, {\rm rank}(R_k) = r_k,  \lVert E_k \rVert_F =\varepsilon_k,  k=1,\cdots,d-1.
		\end{equation}
	TT-SVD computes a tensor $ \mathcal{B} $ in the TT-format with TT-rank $\mathbf{r}=(r_1,\cdots,r_d-1) $ and
	\begin{equation}
		\lVert \mathcal{A}-\mathcal{B}\rVert_F\leq\sqrt{\sum_{k=1}^{d-1}\varepsilon_k^2} \ .
	\end{equation}
\end{theorem}

\subsection{Randomized TT-SVD}
	
When dealing with large-scale data, the computational cost of deterministic SVD is excessively expensive, rendering classical TT-SVD algorithms time-consuming. Randomized low-rank matrix approximation algorithms can approximate large-scale matrices with low-rank matrices while preserving most of the original information. This approach can significantly reduce the processing time of large-scale matrices and enhance computational efficiency. Halko et al. \cite{halko2011} proposed R-SVD for matrices, with arithmetic cost $ \mathcal{O}(rmn+r^2(m+n))$. The procedures of R-SVD are summarized in Algorithm \ref{alg:R-SVD}.

	\begin{algorithm}[!h]
		\caption{R-SVD}
		\label{alg:R-SVD}
		\renewcommand{\algorithmicrequire}{\textbf{Input:}}
		\renewcommand{\algorithmicensure}{\textbf{Output:}}
		\begin{algorithmic}[1]
			\REQUIRE Matrix $ A\in\mathbb{R}^{m\times n} $, target rank $r$, and the oversampling parameter ${p>0}$ 
			\ENSURE Low-rank approximation matrix $ \hat{A}=\hat{U}\hat{S}\hat{V}^\top$ of $A$   
			\STATE Create random Gaussian matrices $ \Omega\in\mathbb{R}^{n\times {(r+p)}} $
			\STATE Calculate $Y=A\Omega $
			\STATE $ [Q,\backsim] = {\tt qr}(Y,0) $
			\STATE $ B=Q^\top A $
			\STATE $ (U,S,V^\top) = {\tt svd}(B)$
			\STATE $ \hat{U}=QU(:,1:r)$, $ \hat{S}=S(1:r,1:r), \hat{V}=V(:,1:r) $
            \RETURN $ \hat{A}=\hat{U}\hat{S}\hat{V}^\top$
			
		\end{algorithmic}
	\end{algorithm}

For a rank-$k$ matrix $\mathbf{A}\in\mathbb{R}^{m\times n}$, its thin-QR decomposition is defined as

\[
[\mathbf{Q},\mathbf{R}] = \texttt{qr}(\mathbf{A}),
\]

where $\mathbf{Q}\in\mathbb{R}^{m\times k}$ and $\mathbf{R}\in\mathbb{R}^{k\times n}$.  
The matrix $\mathbf{Q}$ is column-orthogonal and constitutes an orthonormal basis for the column space of $\mathbf{A}$.

Huber et al. \cite{huber2017} applied the R-SVD algorithm directly to the TT-SVD algorithm, introducing a randomized version of TT-SVD, i.e., TT-rSVD, see Algorithm \ref{alg:TTRS}. The main idea is to enhance computational efficiency by utilizing low-rank approximations of auxiliary unfolding matrices, with arithmetic cost $ \mathcal{O}(r^2\sum_{i=1}^{d-1}n^{d-i+1}) $. The error bound is demonstrated in Theorem \ref{Thm2} below.
	
	\begin{theorem}
		\label{Thm2}
		(\cite{huber2017}, Theorem 2). Given $ \mathcal{A}\in\mathbb{R}^{n_1\times n_2\times\cdots\times n_d} $ and $ s=r+p $ with $ p\ge 4 $. For every $ u,t\ge 1$, the error of the TT-rSVD satisfies 
		\begin{equation}
			\lVert \mathcal{A}-\mathcal{G}_1\times_3^1\mathcal{G}_2\times_3^1\cdots\times_3^1\mathcal{G}_d\rVert_2\leq\sqrt{d-1}\eta(r,p)\min\limits_{{\rm rank}_{\rm TT}(\mathcal{B})\leq(r_1,\cdots,r_{d-1})}\lVert \mathcal{A}-\mathcal{B}\rVert_2,
		\end{equation}
		with probability at least $ (1-5t^{-p}-2e^{-u^2/2})^{d-1} $. The parameter $ \eta $ is given as
		\begin{equation}
			\eta=1+t\sqrt{\frac{12r}{p}}+ut\frac{e\sqrt{r+p}}{p+1}.
		\end{equation}
	\end{theorem}

	\begin{algorithm}[!h]
	\caption{TT-rSVD}
	\label{alg:TTRS}
	\renewcommand{\algorithmicrequire}{\textbf{Input:}}
	\renewcommand{\algorithmicensure}{\textbf{Output:}}
	\begin{algorithmic}[1]
		\REQUIRE Tensor $ \mathcal{A}\in\mathbb{R}^{n_1\times n_2\times\cdots\times n_d} $, target rank $ \mathbf{r} =(r_1,\cdots,r_{d-1}) $, oversampling parameter ${p>0} $, and $ r_0=1 $ 
		\ENSURE Cores $ \mathcal{G}_1,\cdots,\mathcal{G}_d $   
		
		\STATE  $ A :=$ {\tt reshape}$(\mathcal{A},[r_0 n_1,\frac{{\tt numel}(\mathcal{A})}{r_0 n_1}])$
		\FOR{$k = 1$ to $d-1$}
		
		\STATE Create random Gaussian matrix $ \Omega\in\mathbb{R}^{(n_{k+1}\cdot\cdot\cdot n_d)\times{(r_k+p)}} $
		\STATE Calculate $ Y=A\Omega $
		\STATE $ [Q,\backsim] = {\tt qr}(Y,0) $
		\STATE $ Q=Q(:,1:r_k) $
		\STATE New core $ \mathcal{G}_k=$ {\tt reshape}$ (Q,[r_{k-1},n_k,r_k])$
		\STATE Update $\mathcal{A}=\mathcal{A} \times_1 Q^\top$
		\ENDFOR
		\STATE $ \mathcal{G}_d=$reshape$ (\mathcal{A},[r_{d-1},n_d,r_d]) $
		\RETURN $ \mathcal{G}_1 ,\mathcal{G}_2 , \cdots , \mathcal{G}_d $
	\end{algorithmic}
\end{algorithm}

\section{A two-sided sketching algorithm for TT approximation}
	\label{section:3}
As depicted in Algorithm \ref{alg:R-SVD}, although R-SVD can effectively obtain low-rank approximations of matrices, it requires us to revisit the original data matrix. Assuming $A$ represents an enormous matrix, too large to be stored solely within the core memory. The expense associated with data transmission might be so substantial that we are constrained to loading the matrix into the core memory only once. Tropp et al. \cite{tropp2017} introduced a two-sided sketching algorithm for matrices, which can efficiently obtain low-rank approximations of matrices while reading the original data only once. 

Randomized algorithms have proven effective in swiftly obtaining low-rank approximations. However, their efficacy diminishes when faced with matrices exhibiting slow singular value decay, yielding suboptimal results. Power iteration can enhance the rate of decay of singular values, thereby reducing the tail energy and improving the accuracy. From \cite{tropp2017}, it can be inferred that the error of the two-sided sketching algorithm depends on the tail energy of the matrix. Dong et al. \cite{dong2023} proposed an algorithm that combines the subspace power iteration technique with two-sided sketching, i.e., sub-SKETCH, see Algorithm \ref{alg:sub}. When $A$ is dense, the arithmetic cost of Algorithm \ref{alg:sub} is $ \mathcal{O}((q+1)(k+l)mn+kl(m+n)) $ flops.
	
				\begin{algorithm}[!h]
					\caption{sub-SKETCH}
					\label{alg:sub}
					\renewcommand{\algorithmicrequire}{\textbf{Input:}}
					\renewcommand{\algorithmicensure}{\textbf{Output:}}
					\begin{algorithmic}[1]
							\REQUIRE Matrix ${A}\in\mathbb{R}^{m\times n} $,  sketch size parameters $ k,l${, oversampling parameter $ p>0 $}, and the power iteration $ q>0 $
							\ENSURE Rank-$k$ approximation $ \hat{{A}}=QX $ of ${A}$
							\STATE Create random Gaussian matrices $ \Omega\in\mathbb{R}^{n\times {(k+p)}} $
							and $ \Psi\in\mathbb{R}^{l\times m}$
							\STATE $ Y=A\Omega , W=\Psi A$
							\STATE $ [Q_0,\backsim] = {\tt qr}(Y,0)$
							\FOR{$j = 1$ to $q$}
							\STATE $ \hat{{Y}}_j = A^\top Q_{j-1} $
							\STATE $ (\hat{{Q}}_j,\sim) = {\tt qr}(\hat{{Y}}_j,0) $
							\STATE $ {Y}_j=A\hat{Q}_j $
							\STATE $ (\hat{{Q}}_j,\sim) = {\tt qr}({Y}_j,0) $
							\ENDFOR
							\STATE $ Q=Q_q(:,1:k)$
							\STATE $ X=(\Psi Q)\backslash W $
                            \RETURN $\hat{{A}}=QX$
						\end{algorithmic}
				\end{algorithm}
We first present the following theorem characterizing the error of Algorithm \ref{alg:sub} in the spectral norm. Rather than a probabilistic bound with a free confidence parameter $\delta$, we establish an expected error bound, which yields a cleaner constant that directly reflects the average-case performance of the algorithm.
	\begin{theorem}
		\label{Thm3}
		(Average spectral error for Algorithm \ref{alg:sub}). 
        For $A\in\mathbb{R}^{m\times n}$, assume that $ m>n $ and the sketch size parameter satisfies $l>k+1$. Draw random test matrices $ \Omega\in\mathbb{R}^{n\times k} $ and $ \Psi\in\mathbb{R}^{l\times m} $ independently of the standard normal distribution. The rank-$k$ approximation $ \hat{A} $ obtained from Algorithm \ref{alg:sub} satisfies
	\begin{equation}
	\begin{split}
		\mathbb{E}\lVert A-\hat{A}\rVert_2^2\leq & \Big [1+\frac{e^2l}{(l-k)^2} \big(\sqrt{l}+\sqrt{n-k}\big)^2 \Big] \\
		& \Big[1+\sqrt{\frac{k}{l-k-1}}+\frac{e\sqrt{l}}{l-k} \sqrt{n-k} \Big]^{2/(2q+1)}\sigma_{k+1}^2(A) \ .
	\end{split}
	\end{equation}
	\end{theorem}
    
\begin{proof}
We decompose the approximation error into two distinct components via the triangle inequality, i.e.,
	\begin{equation}
		\label{10}
		\begin{split}
			\lVert A-\hat{A}\rVert_2^2= &\lVert (A-QQ^\top A) + (QQ^\top A - QX)\rVert_2^2\\
			&\leq\lVert A-QQ^\top A\rVert_2^2+\lVert X-Q^\top A\rVert_2^2 \ .
		\end{split}
\end{equation}
Using $Q$ computed from Algorithm \ref{alg:sub}, construct a matrix $ P\in\mathbb{R}^{n\times (n-k)} $ with orthonormal columns satisfying
	\begin{equation}
		\label{11}
		PP^\top = I-QQ^\top.
	\end{equation}
Define matrices
\begin{equation}
	\Psi_1:=\Psi P\in\mathbb{R}^{l\times (n-k)}, \quad \Psi_2:=\Psi Q\in\mathbb{R}^{l\times k}.
\end{equation}
From \cref{Lem8}, we have 
\begin{equation}
	X-Q^\top A = \Psi_2^{\dagger}\Psi_1(P^\top A)\ .
\end{equation}
Therefore,
\begin{equation}
	\label{14}
	\begin{split}
	    \mathbb{E}\lVert X-Q^\top A\rVert_2^2 = & \mathbb{E}\lVert \Psi_2^{\dagger}\Psi_1(P^\top A)\rVert_2^2 \\
	    \leq & \mathbb{E}\lVert \Psi_2^{\dagger}\rVert_2^2 \cdot \mathbb{E}\lVert \Psi_1\rVert_2^2 \cdot \mathbb{E}\lVert P^\top A\rVert_2^2 \ .
    \end{split}
\end{equation}
Since $ \Psi_1 $ and $ \Psi_2 $ are Gaussian random matrices, by  \cite{halko2011} (Proposition A.4), we have
\begin{equation}
	\label{15}
	\mathbb{E}\lVert\Psi_2^{\dagger}\rVert_2<\frac{e\sqrt{l}}{l-k} \ .
\end{equation}
Furthermore, since $ P $ and $ Q $ have mutually orthogonal column spaces (i.e., $ P^\top Q = 0 $), the Gaussian random matrices $ \Psi_1 = \Psi P $ and $ \Psi_2 = \Psi Q $ are statistically independent. Following \cite{vershynin2010}(Theorem 5.32), the expected spectral norm of $ \Psi_1 \in \mathbb{R}^{l \times (n-k)} $ is bounded by
	\begin{equation}
		\label{16_new}
		\mathbb{E}\lVert\Psi_1\rVert_2 \leq \sqrt{n-k}+\sqrt{l} \ .
\end{equation}

Combining (\ref{11}), (\ref{14}), (\ref{15}) and (\ref{16_new}), and utilizing the independence of $\Psi_1$ and $\Psi_2$, it can be shown that
\begin{equation}
	\begin{split}
		\mathbb{E}\lVert X-Q^\top A\rVert_2^2\leq & \frac{e^2l}{(l-k)^2} \big (\sqrt{l}+\sqrt{n-k} \big)^2 \mathbb{E}\lVert P^\top A\rVert_2^2 \\
		\leq& \frac{e^2l}{(l-k)^2} \big (\sqrt{l}+\sqrt{n-k} \big)^2 \mathbb{E}\lVert A-QQ^\top A\rVert_2^2 \\
		\leq& \frac{e^2l}{(l-k)^2} \big(\sqrt{l}+\sqrt{n-k} \big)^2 \\
		& \Big[1+\sqrt{\frac{k}{l-k-1}}+\frac{e\sqrt{l}}{l-k}\cdot \sqrt{n-k} \Big]^{2/(2q+1)} \sigma_{k+1}^2(A) \ .
	\end{split}
\end{equation}
The last inequality holds due to the average spectral error for the power scheme \cite{halko2011} (Corollary 10.10). Combining this result with Eq. (\ref{10}), we complete the proof.
\end{proof}

\subsection{Subspace power iteration for TT approximation}

By integrating the sub-SKETCH (Algorithm \ref{alg:sub}) with the TT-rSVD (Algorithm \ref{alg:TTRS}), we propose an effective sketch TT algorithm, referred to as TT-subSKETCH. The detailed procedure of this algorithm is delineated in Algorithm \ref{alg:TTPS}, accompanied by its error analysis in Theorem \ref{Thm4}.

\begin{algorithm}[!h]
	\caption{TT-subSKETCH}
	\label{alg:TTPS}
	\renewcommand{\algorithmicrequire}{\textbf{Input:}}
	\renewcommand{\algorithmicensure}{\textbf{Output:}}
	\begin{minipage}{\linewidth}
		\small
		\begin{algorithmic}[1]
			\REQUIRE Tensor $\mathcal{A}\in\mathbb{R}^{n_1\times\cdots\times n_d}$, 
			target rank $\mathbf{r}=(r_1,\dots,r_{d-1})$, 
			sketch size $\mathbf{l}=(l_1,\dots,l_{d-1})$, 
			$r_0=1$, oversampling $p>0$, power iteration $q\ge 0$
			\ENSURE Cores $\mathcal{G}_1,\dots,\mathcal{G}_d$ of the TT-approximation
			
			\STATE $A^{(1)}:=\texttt{reshape}(\mathcal{A},[r_0 n_1,\frac{\texttt{numel}(\mathcal{A})}{r_0 n_1}])$
			\FOR{$k = 1$ to $d-1$}
			\STATE Create random Gaussian matrices:
			\STATE \quad $\Omega^{(k)}\in\mathbb{R}^{(n_{k+1}\cdots n_d)\times (r_k+p)}$
			\STATE \quad $\Psi^{(k)}\in\mathbb{R}^{l_k\times (r_{k-1} n_k)}$
			\STATE $Y^{(k)}=A^{(k)}\Omega^{(k)},\quad W^{(k)}=\Psi^{(k)} A^{(k)}$
			\STATE $[Q_0^{(k)},\cdot]=\texttt{qr}(Y^{(k)},0)$
			\FOR{$j = 1$ to $q$}
			\STATE $\hat{Y}_j^{(k)} = (A^{(k)})^\top Q_{j-1}^{(k)}$
			\STATE $(\hat{Q}_j^{(k)},\cdot)=\texttt{qr}(\hat{Y}_j^{(k)},0)$
			\STATE $Y_j^{(k)} = A^{(k)} \hat{Q}_j^{(k)}$
			\STATE $(Q_j^{(k)},\cdot)=\texttt{qr}(Y_j^{(k)},0)$
			\ENDFOR
			\STATE $Q^{(k)}=Q^{(k)}_q(:,1:r_k)$
			\STATE $X^{(k)}=(\Psi^{(k)} Q^{(k)})\backslash W^{(k)}$
			\STATE $\mathcal{G}_k=\texttt{reshape}(Q^{(k)},[r_{k-1},n_k,r_k])$
			\STATE $A^{(k+1)}=\texttt{reshape}(X^{(k)},[r_k n_{k+1},\frac{\texttt{numel}(X^{(k)})}{r_k n_{k+1}}])$
			\ENDFOR
			\STATE $\mathcal{G}_d=X^{(d-1)}$
			\RETURN $\mathcal{G}_1,\dots,\mathcal{G}_d$
		\end{algorithmic}
	\end{minipage}
\end{algorithm}

\textit{Remark}. When \( q = 0 \), i.e., without power iteration, our TT-subSKETCH algorithm reduces to a one-pass sketching algorithm, which only accesses the original data once and thus avoids excessive data transmission costs.


\begin{theorem}\label{Thm4}
	Let $ \hat{\mathcal{A}} $ be the TT approximation of a tensor $ \mathcal{A}\in\mathbb{R}^{n_1\times n_2\times\cdots\times n_d} $ by TT-subSKETCH (i.e., Algorithm \ref{alg:TTPS}) with the target TT-rank $ \mathbf{r}=(r_1,\cdots,r_{d-1}) $, $ (r_0=r_d=1) $, sketch size $ \mathbf{l}=(l_1,\cdots,l_{d-1})$, and the power iteration q satisfied
	\begin{equation}
		\begin{split}
			\mathbb{E}\lVert\mathcal{A}-\hat{\mathcal{A}}\rVert_2^2\leq 
			& \sum_{k=1}^{d-1} \Big[1+\frac{e^2l_k}{(l_k-r_k)^2} \big(\sqrt{l_k}+\sqrt{\min(r_k n_{k+1},n_{k+2}\cdots n_d)-r_k} )^2\Big]\\ 
			&\Big[1+\sqrt{\frac{r_k}{l_k-r_k-1}} + \frac{e\sqrt{l_k}}{l_k-r_k}\sqrt{\min(r_k n_{k+1},n_{k+2}\cdots n_d)-r_k} \Big]^{2/(2q+1)}
			\\&\sigma_{r_k+1}^2(A^{(k)}) .
		\end{split}
	\end{equation}
\end{theorem}
\begin{proof}
    {Since $ X^{(k)} (k\in [1,\cdots,d-1]) $ is  naturally associated with $A^{(k+1)}$, it will continue to decompose further in the Algorithm \ref{alg:TTPS}, which means that $X^{(k)}$ is approximated by some matrix $ \hat{X}^{(k)} $. Let $ \hat{A}^{(k)}=Q^{(k)}\hat{X}^{(k)} $ be an approximated matrix of $ A^{(k)} $. } 
\begin{equation}
	\label{equation:10}
	\begin{split}
		\mathbb{E}\lVert\mathcal{A}-\hat{\mathcal{A}}\rVert_2 = &\mathbb{E}\lVert A^{(1)}-\hat{A}^{(1)}\rVert_2 \\
        = & \mathbb{E}\lVert A^{(1)}-Q^{(1)}\hat{X}^{(1)}\rVert_2 \\
		=&\mathbb{E}\lVert A^{(1)}-Q^{(1)}X^{(1)}+Q^{(1)}(X^{(1)}-\hat{X}^{(1)})\rVert_2 \\
		\leq&\mathbb{E}\lVert A^{(1)}-Q^{(1)}X^{(1)}\rVert_2 + \mathbb{E}\lVert Q^{(1)}(X^{(1)}-\hat{X}^{(1)})\rVert_2 \\
		=&\mathbb{E}\lVert A^{(1)}-Q^{(1)}X^{(1)}\rVert_2 + \mathbb{E}\lVert X^{(1)}-\hat{X}^{(1)}\rVert_2 \ .
	\end{split}
\end{equation}
The last relation holds due to the unitary invariance of the {spectral-norm}, since $ Q^{(k)}, k\in [1,\cdots,d-1]$, has orthonormal columns. Moreover, $ \forall k\in [1,\cdots,d-1] $, there holds
\begin{equation}
	\label{equation:11}
	\begin{split}
		\mathbb{E}\lVert X^{(k)}-\hat{X}^{(k)}\rVert_2 = 
        & \mathbb{E}\lVert A^{(k+1)}-Q^{(k+1)}\hat{X}^{(k+1)}\rVert_2 \\
		\leq & \mathbb{E}\lVert A^{(k+1)}-Q^{(k+1)}X^{(k+1)}\rVert_2 + \mathbb{E}\lVert X^{(k+1)}-\hat{X}^{(k+1)}\rVert_2 \ .
	\end{split}
\end{equation}
Inserting this iteratively into (\ref{equation:10}) gives
\begin{equation}
	\label{equation:12}
	\begin{split}
		\mathbb{E}\lVert\mathcal{A}-\hat{\mathcal{A}}\rVert_2\leq
		\sum_{k=1}^{d-1}\mathbb{E}\lVert {A^{(k)}-\hat{A}^{(k)}}\rVert_2 \ .
	\end{split}
\end{equation}
Combining with \cref{Thm3}, we complete the proof.
\end{proof}

Furthermore, the computational cost of TT-subSKETCH is dominated by the two-sided sketching and the $q$ subspace power iterations at the first unfolding step. Assuming $n_i \approx n$ and $r_i \approx r$, the overall arithmetic cost scales as $\mathcal{O}((q+1)r^2 n^d)$. Compared to the $\mathcal{O}(n^{d+1})$ complexity of the deterministic TT-SVD, our algorithm provides a substantial theoretical speedup since $(q+1)r^2 \ll n$ strictly holds in typical low-rank scenarios. This theoretically justifies the dramatic reduction in CPU time observed in our numerical experiments.

\subsection{Robustness analysis}

In practical applications, we may only have access to data that are corrupted with noise. Unlike the noise-free setting in Theorems \ref{Thm3} and \ref{Thm4}, where the randomness stems solely from the Gaussian sketch matrices $\Omega^{(k)}$ and $\Psi^{(k)}$ with known distributional properties, the noise matrices $Z^{(k)}$ and $\widetilde{Z}^{(k)}$ here are treated as arbitrary deterministic perturbations with unknown distributions. Consequently, an expected error bound is no longer attainable, and the approximation error must instead be controlled in a high-probability sense with respect to the randomness of the sketch matrices alone. In this scenario, Ma et al.\ \cite{ma2023} conducted a robustness analysis for the recovery of low-rank matrices and low-tubal-rank tensors from noisy sketches. We generalize the result from t-product to TT format in Theorem \ref{Thm5} below.

\begin{theorem}
    \label{Thm5}
Let $ \mathcal{A}\in\mathbb{R}^{n_1\times n_2\times\cdots\times n_d} $ be an initial tensor. For each k we have $ r_k<l_k $, and consider the sketches in Algorithm \ref{alg:TTPS}, i.e.,
	\begin{equation}
		 Y^{(k)}=A^{(k)}\Omega^{(k)} +Z^{(k)},  \ \ W^{(k)}=\Psi^{(k)} A^{(k)}+\widetilde{Z}^{(k)}, 
	\end{equation}
where $ \Psi^{(k)}\in\mathbb{R}^{l_k\times (r_{k-1} n_k)} $ and $ \Omega^{(k)}\in\mathbb{R}^{(n_{k+1}\cdots n_d)\times r_k} $ are independent Gaussian random matrices, and $ Z^{(k)} $ and $ \widetilde{Z}^{(k)} $ are noise matrices originated from k sketching. 
For any $ \delta_1,\delta_2,\delta_3>0 $, with probability at least $ 1-\delta_1-\delta_2-\delta_3 $, the TT approximation $ \hat{\mathcal{A}} $ obtained by TT-subSKETCH (i.e., Algorithm \ref{alg:TTPS}) satisfies
	\begin{equation}
		\begin{split}
        \lVert \mathcal{A}-\hat{\mathcal{A}}\rVert_2 
        \leq & \sum_{k=1}^{d-1}\Big[\frac{r_k\sqrt{(r_{k-1}n_k-r_k)}\lVert Z^{(k)}\lVert_2}{\sqrt{2\delta_1 \log(1/(1-\delta_2))+\delta_1}-\sqrt{\delta_1}}\\
        & + \frac{\lVert (\widetilde{Z}^{(k)})\lVert_2}{\sqrt{l_k}-\sqrt{r_k}-\sqrt{2\log(2/\delta_3)}}\Big] \ .
		\end{split}
	\end{equation}
\end{theorem}

\begin{proof}
From Theorem \ref{Thm4}, we know that the TT approximation error originates from the sketching error of the auxiliary unfolding matrices. Therefore, we first consider the
$k-th$ sketching of the unfolding matrix, i.e.,
    \begin{align}
    \begin{split}
    \label{eq24}
    \hat{A}^{(k)}&=Y^{(k)}(\Psi^{(k)}Y^{(k)})\backslash W^{(k)} \\
    &=\widetilde{A}^{(k)}+Y^{(k)}(\Psi^{(k)}Y^{(k)})^{\dagger}(\widetilde{Z}^{(k)}) \ ,
    \end{split}
    \end{align}
where $ \widetilde{A}^{(k)}:=Y^{(k)}(\Psi^{(k)}Y^{(k)})^{\dagger}(\Psi^{(k)}A^{(k)}) $. Denote the SVD of $ Y^{(k)} $ as 
    \begin{align}
        Y^{(k)}=U_Y^{(k)}\Sigma_Y^{(k)}{V_Y^{(k)}}^\top, 
    \end{align}
where $ U_Y^{(k)}\in\mathbb{R}^{(r_{k-1}n_k)\times r_k}, \Sigma_Y^{(k)}\in\mathbb{R}^{r_k\times r_k}$, and ${V_Y^{(k)}}^\top \in\mathbb{R}^{r_k\times r_k}. $ Since $ \Psi^{(k)}U_Y^{(k)} $ has linearly independent columns \cite{ma2023}, 
    \begin{align}
    \begin{split}
        Y^{(k)}(\Psi^{(k)}Y^{(k)})^{\dagger}&=U_Y\Sigma_YV_Y^\top (\Psi^{(k)}U_Y\Sigma_YV_Y^\top )^{\dagger}\\
        & = U_Y^{(k)}\Sigma_Y^{(k)}{V_Y^{(k)}}^\top V_Y^{(k)}{\Sigma_Y^{(k)}}^{-1}(\Psi^{(k)}U_Y^{(k)})^{\dagger}\\
        &=U_Y^{(k)}(\Psi^{(k)}U_Y^{(k)})^{\dagger} \ .
    \end{split}
    \end{align}
Thus, Eq. \eqref{eq24} can be simplified to
    \begin{align}
        \hat{A}^{(k)}=\widetilde{A}^{(k)}+U_Y^{(k)}(\Psi^{(k)}U_Y^{(k)})^{\dagger}(\widetilde{Z}^{(k)}) \ .
    \end{align} 
Our goal is to bound the approximation error, i.e.,
    \begin{align}
    \begin{split}
    \label{eq28}
        \lVert \hat{A}^{(k)}-A^{(k)}\lVert_2&=\lVert\widetilde{A}^{(k)}-A+U_Y^{(k)}(\Psi^{(k)}U_Y^{(k)})^{\dagger}(\widetilde{Z}^{(k)})\lVert_2\\
        &\leq \lVert \widetilde{A}^{(k)}-A\lVert_2+\lVert U_Y^{(k)}(\Psi^{(k)}U_Y^{(k)})^{\dagger}(\widetilde{Z}^{(k)})\lVert_2.
    \end{split}
    \end{align}
For the first term, let $P^{(k)}:=I-U_Y^{(k)}(\Psi^{(k)}U_Y^{(k)})^{\dagger}\Psi^{(k)}$, and then we have
    \begin{align}
        \begin{split}
            \widetilde{A}^{(k)}-A &= U_Y^{(k)}(\Psi^{(k)}U_Y^{(k)})^{\dagger}\Psi^{(k)}A^{(k)}-A^{(k)}\\
            &=-[I-U_Y^{(k)}(\Psi^{(k)}U_Y^{(k)})^{\dagger}\Psi^{(k)}]A^{(k)} \\
            &=  - P^{(k)}A^{(k)} \ .
        \end{split}
    \end{align} 
We also observe that $ P^{(k)} $ is a projection, i.e., $ {P^{(k)}}^2=P^{(k)} $, satisfying
    \begin{align}
        \begin{split}
            \ker P^{(k)}=U_Y^{(k)}, \ \  {\rm Im} P^{(k)} := S^{(k)}.
        \end{split}
    \end{align}
From Lemma \ref{Lem3} in the Appendix, we have
    \begin{align}
        P^{(k)}A^{(k)}=S^{(k)}(T^{(k)} S^{(k)})^{-1}T^{(k)} A^{(k)},
    \end{align}
where $ T^{(k)}:={{U_Y^{(k)}}_\perp^\top} $ represents the transpose of the orthogonal complement of $ U_Y^{(k)} $. With the SVD of $ A^{(k)} $ being $ A^{(k)}=U_A^{(k)}\Sigma_A^{(k)}{V_A^{(k)}}^\top $,  we then can bound 
    \begin{align}
        \begin{split}
        \label{eq32}
            \lVert \widetilde{A}^{(k)}-A\lVert&=\lVert P^{(k)}A^{(k)}\lVert_2\\
            &=\lVert S^{(k)}(T^{(k)S^{(k)}})^{-1}T^{(k)}U_A^{(k)}\Sigma_A^{(k)}{V_A^{(k)}}^\top \lVert_2\\
            &\leq\lVert (T^{(k)}S^{(k)})^{-1}\lVert_2 \lVert T^{(k)}U_A^{(k)}\Sigma_A^{(k)}{V_A^{(k)}}^\top \lVert_2 \ .
        \end{split}
    \end{align}
Since $ Y^{(k)}=A^{(k)}\Omega^{(k)}+Z^{(k)}=U_A^{(k)}\Sigma_A^{(k)}{V_A^{(k)}}^\top \Omega^{(k)}+Z^{(k)} $, we have
    \begin{align}
    \label{eq33}
        U_A^{(k)}\Sigma_A^{(k)}{V_A^{(k)}}^\top \Omega^{(k)}=Y^{(k)}-Z^{(k)}.
    \end{align}
where $ {V_A^{(k)}} $ has orthonormal columns and $\Omega^{(k)}$ is Gaussian matrix. With probability~1, $ {V_A^{(k)}}^\top \Omega^{(k)} $ has linearly independent rows \cite{ma2023}, i.e.,
    \begin{align}
    \label{eq34}
        ({V_A^{(k)}}^\top \Omega^{(k)})({V_A^{(k)}}^\top \Omega^{(k)})^\dagger=I \ .
    \end{align}
Combining (\ref{eq33}) and (\ref{eq34}) yields
    \begin{align}
        U_A^{(k)}\Sigma_A^{(k)}=(Y^{(k)}-Z^{(k)})({V_A^{(k)}}^\top \Omega^{(k)})^\dagger.
    \end{align}
Simplify the second term of the right hand side of (\ref{eq32}), {Since $ T^{(k)}:={{U_Y^{(k)}}_\perp}^\top $ and $Y^{(k)}=U_Y^{(k)}\Sigma_Y^{(k)}{V_Y^{(k)}}^\top$, we know that $T^{(k)}Y^{(k)}=0$.}
    \begin{align}
        \begin{split}
        \label{eq36}
            \lVert T^{(k)}U_A^{(k)}\Sigma_A^{(k)}{V_A^{(k)}}^\top \lVert_2
            &= \lVert T^{(k)}U_A^{(k)}\Sigma_A^{(k)}\lVert_2 \\
            &= \lVert T^{(k)}(Y^{(k)}-Z^{(k)})({V_A^{(k)}}^\top \Omega^{(k)})^\dagger\lVert_2\\
            &=\lVert T^{(k)}Z^{(k)}({V_A^{(k)}}^\top \Omega^{(k)})^\dagger\lVert_2\\
            &\leq \lVert T^{(k)}Z^{(k)}\lVert_2\lVert({V_A^{(k)}}^\top \Omega^{(k)})^\dagger\lVert_2\\
            &=\frac{\lVert T^{(k)}Z^{(k)}\lVert_2}{\sigma_{\min}({V_A^{(k)}}^\top \Omega^{(k)})}\\
            &=\frac{\lVert Z^{(k)}\lVert_2}{\sigma_{\min}({V_A^{(k)}}^\top \Omega^{(k)})}.
        \end{split}
    \end{align}
The above first and last equations are derived from the unitary invariance property. Using the fact that $ \lVert (T^{(k)}S^{(k)})^{-1}\lVert_2=1/\sigma_{\min}(T^{(k)}S^{(k)}) $ for the first term of (\ref{eq32}), we obtain the following bound
    \begin{align}
        \lVert \widetilde{A}^{(k)}-A\lVert\leq\frac{\lVert Z^{(k)}\lVert_2}{\sigma_{\min}(T^{(k)}S^{(k)})\sigma_{\min}({V_A^{(k)}}^\top \Omega^{(k)})}.
    \end{align}  
We now derive a probabilistic bound from (\ref{eq36}) using concentration inequalities from random matrix theory. Since $ S^{(k)} $ can be seen as the $ (r_{k-1}n_k)\times (r_{k-1}n_k-r_k) $ submatrix of a Haar unitary matrix $ (S^{(k)},S^{(k)}_\perp) $, by unitary invariance property, we have
\begin{equation}
     M^{(k)}:= \Big[\begin{smallmatrix}
    T^{(k)}\\
    {U^{(k)}_Y}^\top 
    \end{smallmatrix} \Big](S^{(k)},S^{(k)}_\perp),
\end{equation}
where $M^{(k)}$ is also a Haar unitary matrix \cite{ma2023}, and $ T^{(k)}S^{(k)} $ is exactly the upper left $(r_{k-1}n_k-r_k)\times (r_{k-1}n_k-r_k) $ corner of $M^{(k)}$. 
Applying Lemma \ref{Lem4} in the Appendix, we have
    \begin{equation}
        \mathbb{P}\big(\sigma_{\min}(T^{(k)}S^{(k)})\geq\frac{\sqrt{\delta_1}}{\sqrt{r_k(r_{k-1}n_k-r_k)}}\big)\geq1-\delta_1 \ ,
    \end{equation}
for $\forall\delta_1>0$. Since $ {V_A^{(k)}}^\top \Sigma^{(k)} $ is distributed as an $r_k\times r_k$ Gaussian random matrix, by Lemma \ref{Lem5} in the Appendix, for any $0<\delta_2<1$,
    \begin{equation}
        \mathbb{P}\big(\sigma_{\min}({V_A^{(k)}}^\top \Sigma^{(k)})\geq\frac{\sqrt{2\log(1/(1-\delta_2))+1}-1}{\sqrt{r_k}} \big) = 1-\delta_2 \ .
    \end{equation}
Combining the two probability estimates, with probability at least $1-\delta_1-\delta_2$,
    \begin{align}
    \label{eq38}
        \lVert \widetilde{A}^{(k)}-A\lVert\leq\frac{r_k\sqrt{(r_{k-1}n_k-r_k)}\lVert Z^{(k)}\lVert_2}{\sqrt{2\delta_1 \log(1/(1-\delta_2))+\delta_1}-\sqrt{\delta_1}} \ .
    \end{align}
    
For the second term of (\ref{eq28}), note that $\Psi^{(k)}U_Y^{(k)}$ is distributed as an $l_k\times r_k$ Gaussian random matrix. Since $l_k>r_k$, by Lemma \ref{Lem6} in the Appendix, $\forall \delta_3\in(0,1]$, with probability at least $1-\delta_3$,
    \begin{align}
    \label{eq39}
        \sigma_{\min}(\Psi^{(k)}U_Y^{(k)})\geq\sqrt{l_k}-\sqrt{r_k}-\sqrt{2\log(2/\delta_3)}.
    \end{align}
    Thus, with probability at least $1-\delta_3$,
    \begin{align}
    \begin{split}
           \label{eq40}
        \lVert U_Y^{(k)}(\Psi^{(k)}U_Y^{(k)})^{\dagger}(\widetilde{Z}^{(k)})\lVert_2&=\lVert (\Psi^{(k)}U_Y^{(k)})^{\dagger}(\widetilde{Z}^{(k)})\lVert_2\\
        &\leq\frac{\lVert (\widetilde{Z}^{(k)})\lVert_2}{\sigma_{\min}(\Psi^{(k)}U_Y^{(k)})}\\
        &\leq \frac{\lVert (\widetilde{Z}^{(k)})\lVert_2}{\sqrt{l_k}-\sqrt{r_k}-\sqrt{2\log(2/\delta_3)}} \ .
    \end{split}
    \end{align}
    
Combining (\ref{eq38}) and (\ref{eq40}), we finally get the error of the $k$-th sketching of unfolding matrix with probability at least $1-\delta_1-\delta_2-\delta_3$, i.e.,
    \begin{align}
    \begin{split}
        \lVert \hat{A}^{(k)}-A^{(k)}\lVert_2 
        \leq & \frac{r_k\sqrt{(r_{k-1}n_k-r_k)}\lVert Z^{(k)}\lVert_2}{\sqrt{2\delta_1 \log(1/(1-\delta_2))+\delta_1}-\sqrt{\delta_1}}\\
        & + \frac{\lVert (\widetilde{Z}^{(k)})\lVert_2}{\sqrt{l_k}-\sqrt{r_k}-\sqrt{2\log(2/\delta_3)}} \ .
    \end{split}
    \end{align}
Finally, by applying the triangle inequality to the tensor unfolding operations
\begin{equation}
        \lVert \mathcal{A}-\hat{\mathcal{A}}\rVert_2
        \leq \sum_{k=1}^{d-1} \lVert \hat{A}^{(k)}-A^{(k)}\lVert_2 \ ,
\end{equation}
we complete the proof.
\end{proof}


\section{Numerical Experiments}
\label{section:4}
To test the effectiveness of the proposed TT-subSKETCH algorithm, we compare it with the most related state-of-the-art algorithms TT-SVD and TT-rSVD. For ease of comparison, we set the power iteration parameter $q$ to 0, 1, and 2 in Algorithm \ref{alg:TTPS} and name them TT-SKETCH, TT-subSKETCH1, and TT-subSKETCH2, respectively. 

The data used for validation include synthetic data and real-world data with and without noise. Regarding the numerical settings, the sketch parameter $ l_n=2r_n, n=1,\cdots,d$. 
The quality of the reconstructed tensor is measured by the peak signal-to-noise ratio (PSNR), say $ \rho $, and the relative error $ \epsilon $. For tensor $ \mathcal{A}\in\mathbb{R}^{n_1\times n_2\times n_3} $ and its low-rank TT approximation $ \hat{\mathcal{A}} $, the PSNR is defined as
\begin{equation}
	\rho=10 \log_{10}\frac{n_1n_2n_3\lVert\hat{\mathcal{A}}\rVert_\infty^2}{\lVert\mathcal{A}-\hat{\mathcal{A}}\rVert_F^2} \ .
\end{equation}

The relative error $\epsilon$ of the low-rank reconstruction tensor is defined as
\begin{equation}
	\epsilon =\lVert\mathcal{A}-\hat{\mathcal{A}}\rVert_F/\lVert\mathcal{A}\rVert_F \ .
\end{equation}
All experiments were run on a laptop with 2.4 GHz Intel Core i7-8700T CPU and 16GB of RAM. The MATLAB Tensor Toolbox \cite{brett2023} is utilized to perform the experiments.

\subsection{Power function tensor}
We first evaluate the performance of different algorithms on power function tensor data. The power function tensor can be defined entry-wise as
\begin{equation}
	x_{i_1 i_2,\cdots,i_n}=\frac{1}{\sqrt[h]{i_1^h+\cdots i_n^h}}, \ n=1,2,\cdots,N.
\end{equation}
We fix $ N=5 $, $ h=5 $, and the size of data to be $ 45\times 45\times 45\times 45\times 45 $, and the equal target rank $(r,r,r,r)$ with $ r\in[2,11] $. 

\begin{figure}[!tb]
	\centering
	\includegraphics[width=0.9\textwidth]{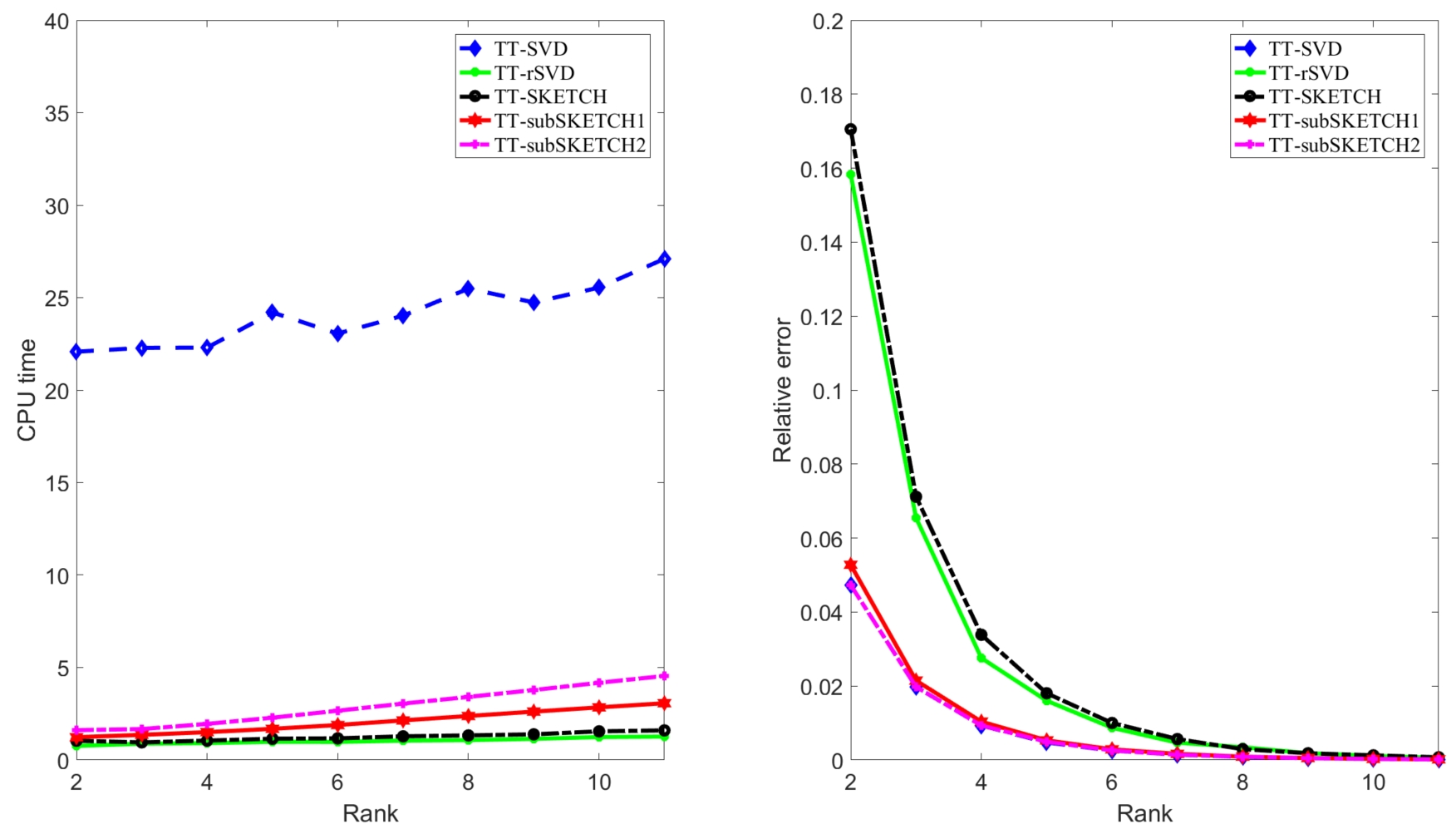}
	\caption{Results comparison on the power function tensor with a size of $ 45\times 45\times 45\times 45\times 45 $ in terms of the CPU time (left) and relative error (right) dependency on the target rank. The oversampling parameter $p$ is fixed at 2. 
    }
	\label{fig:1}
\end{figure}


The results on the power function tensor are given in Figure \ref{fig:1} in terms of the CPU time and relative error corresponding to different target ranks. It indicates that random algorithms consume significantly less CPU time compared to deterministic ones, and the CPU time increases with the power iteration parameter becoming large, especially for larger target ranks. For the relative error shown in the right panel of Figure \ref{fig:1}, we see that when the target rank is relatively low, the power iteration method significantly reduces the error. In particular, the relative errors of algorithms TT-subSKETCH1 and TT-subSKETCH2 are very close to TT-SVD and are significantly lower than TT-rSVD, demonstrating the excellent performance of the proposed algorithms.

\subsection{Real-world data}
%
We now evaluate the performance of different algorithms on a hyperspectral image named PaviaU\footnote{\scriptsize{https://ehu.eus/ccwintco/index.php?title=Hyperspectral$\_$Remote$\_$Sensing$\_$Scenes}}  with a size of $ 610\times 340\times 103 $ and a color video named Driving\footnote{\scriptsize{https://pixabay.com/videos/mercedes-glk-car-test-offroad-suv-1406/}} with a size of $ 360\times 640\times 3\times 100 $. The oversampling parameter $p$ is fixed to 5.

The results are given in Figure \ref{fig:2} on the hyperspectral image and in Figure \ref{fig:4} on the color video in terms of the CPU time, relative error, and PSNR corresponding to different target ranks. In particular, Figure \ref{fig:3} and Figure \ref{fig:5} give the qualitative results of the low-rank approximation of one spectral channel of the hyperspectral image and one frame of the color video, respectively, by different methods in terms of CPU time and PSNR.
The results show that the relative error and PSNR of TT-subSKETCH2 and TT-SVD are very similar, with our TT-subSKETCH2 being far more efficient in terms of CPU time. Although TT-rSVD and TT-SKETCH take less CPU time, their relative error and PSNR are not comparable to TT-subSKETCH2. Furthermore, comparing TT-subSKETCH1 with TT-subSKETCH2, we can see that increasing the number of power iterations can achieve better approximation result at the expense of slightly higher computation time.

\begin{figure}[!tb]
	\centering
	\includegraphics[width=0.9\textwidth,height=6.0cm]{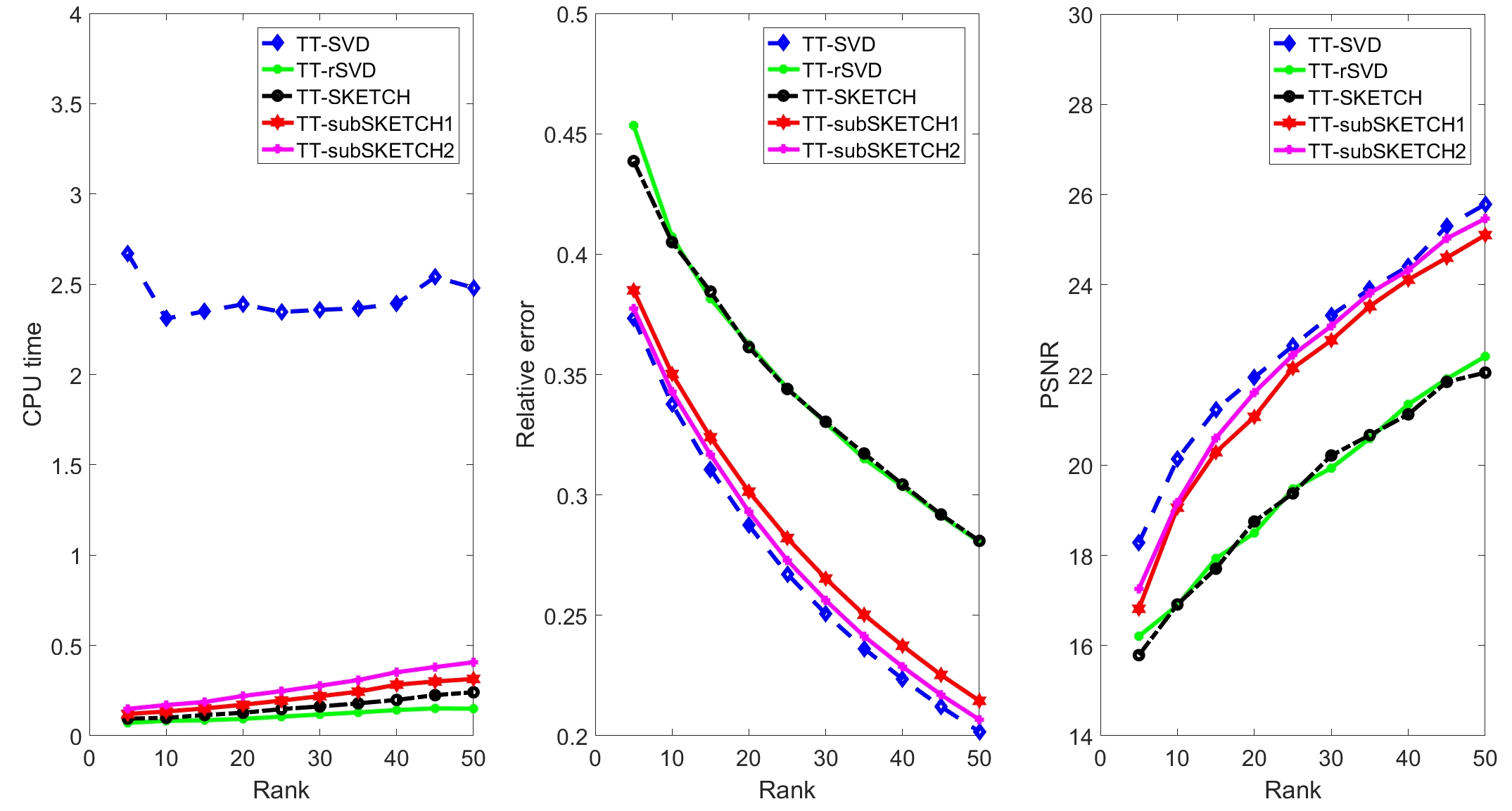}
	\caption{Results comparison on the hyperspectral image with a size of $ 610\times 340\times 103 $ in terms of CPU time (left), relative error (middle), and PSNR (right) dependency on the target rank. The oversampling parameter p is fixed at 5.}
	\label{fig:2}
\end{figure}

\begin{figure}[!tb]
	\centering
	\includegraphics[width=0.9\textwidth,height=6.0cm]{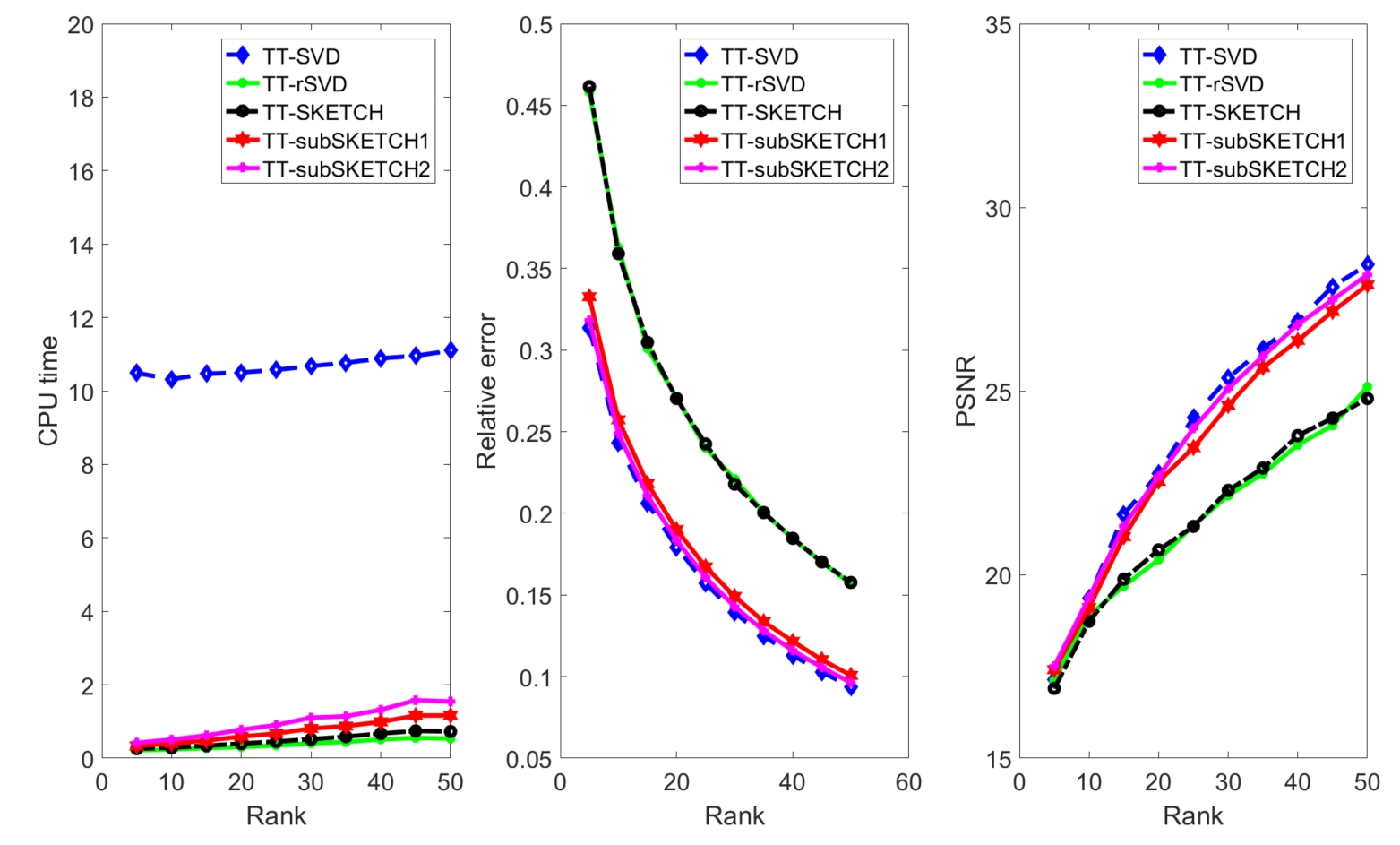}
	\caption{Results comparison on the color video with a size of $ 360\times 640\times 3\times 100 $ in terms of CPU time (left), relative error (middle), and PSNR (right) dependency on the target rank. The oversampling parameter p is fixed at 5.}
	\label{fig:4}
\end{figure}

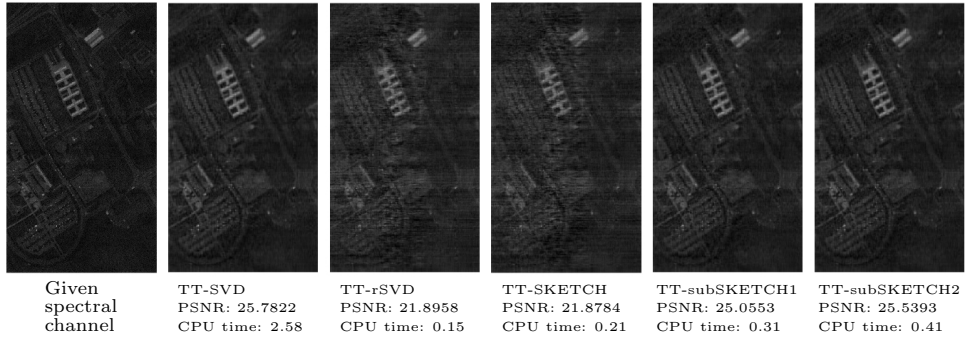
\begin{figure}[!tb]
	\centering
	\includegraphics[trim={{.00\linewidth} {.40\linewidth} {.00\linewidth} {.00\linewidth}}, clip, width=0.99\textwidth]{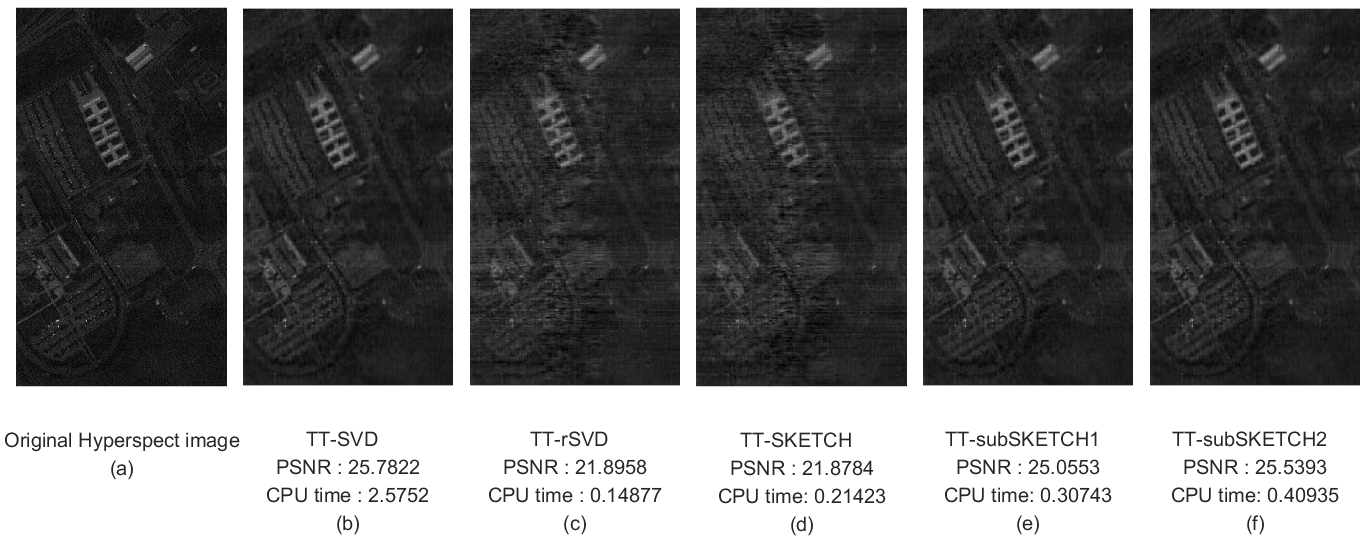}
    \put(-348, -4){\scriptsize Given}
    \put(-348, -11){\scriptsize  spectral}
    \put(-348, -18){\scriptsize   channel}
    \put(-298, -4){\tiny{TT-SVD}}
    \put(-298, -11){\tiny  PSNR: 25.7822}
    \put(-298, -18){\tiny   CPU time: 2.58}
    \put(-237, -4){\tiny{TT-rSVD}}
    \put(-237, -11){\tiny  PSNR: 21.8958}
    \put(-237, -18){\tiny   CPU time: 0.15}
    \put(-176, -4){\tiny{TT-SKETCH}}
    \put(-176, -11){\tiny  PSNR: 21.8784}
    \put(-176, -18){\tiny   CPU time: 0.21}
    \put(-118, -4){\tiny{TT-subSKETCH1}}
    \put(-118, -11){\tiny  PSNR: 25.0553}
    \put(-118, -18){\tiny   CPU time: 0.31}
    \put(-57, -4){\tiny{TT-subSKETCH2}}
    \put(-57, -11){\tiny  PSNR: 25.5393}
    \put(-57, -18){\tiny   CPU time: 0.41}
    
	\caption{Low-rank approximation of one spectral channel of the hyperspectral image by different methods in terms of CPU time and PSNR.}
	\label{fig:3}
\end{figure}

\begin{figure}[!tb]
	\centering
    \begin{tabular}{ccc}
	\includegraphics[trim={{.01\linewidth} {1.35\linewidth} {2.20\linewidth} {.01\linewidth}}, clip, width=0.3\textwidth, height=0.18\textwidth]{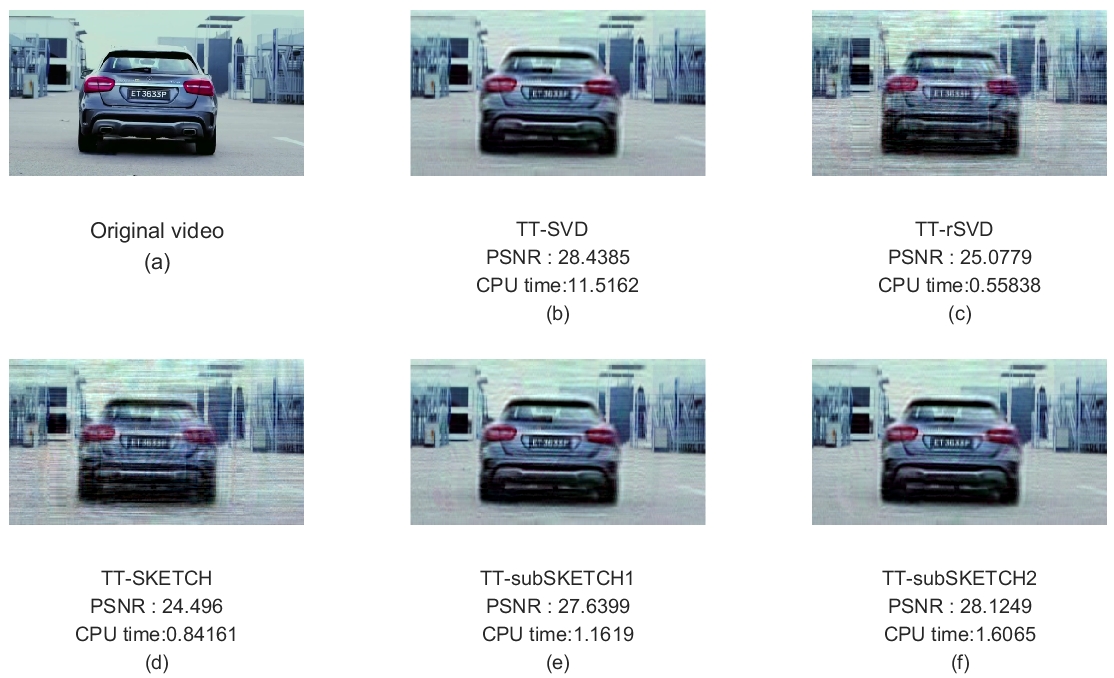} 
    \put(-87, -7){\tiny{Given original frame}} & 
    \includegraphics[trim={{1.12\linewidth} {1.35\linewidth} {1.10\linewidth} {.01\linewidth}}, clip, width=0.3\textwidth, height=0.18\textwidth]{Driving.png}
    \put(-70, -4){\tiny{TT-SVD}}
    \put(-107, -11){\tiny{PSNR: 28.4385; CPU time: 11.52}} & 
    \includegraphics[trim={{2.21\linewidth} {1.35\linewidth} {0.01\linewidth} {.01\linewidth}}, clip, width=0.3\textwidth, height=0.18\textwidth]{Driving.png}
    \put(-70, -4){\tiny{TT-rSVD}}
    \put(-105, -11){\tiny{PSNR: 25.0779; CPU time: 0.56}} \\
    \includegraphics[trim={{.01\linewidth} {.40\linewidth} {2.20\linewidth} {0.95\linewidth}}, clip, width=0.3\textwidth, height=0.18\textwidth]{Driving.png} 
    \put(-73, -4){\tiny{TT-SKETCH}}
    \put(-103, -11){\tiny{PSNR: 24.4960; CPU time: 0.84}} & 
    \includegraphics[trim={{1.12\linewidth} {0.40\linewidth} {1.10\linewidth} {0.95\linewidth}}, clip, width=0.3\textwidth, height=0.18\textwidth]{Driving.png} 
     \put(-81, -4){\tiny{TT-subSKETCH1}}
    \put(-106, -11){\tiny{PSNR: 27.6399; CPU time: 1.16}} & 
    \includegraphics[trim={{2.21\linewidth} {0.40\linewidth} {0.01\linewidth} {.95\linewidth}}, clip, width=0.3\textwidth, height=0.18\textwidth]{Driving.png} 
     \put(-81, -4){\tiny{TT-subSKETCH2}}
    \put(-105, -11){\tiny{PSNR: 28.1249; CPU time: 1.61}}
      
    \end{tabular}
    
	\caption{Low-rank approximation of one frame of the color video clip by different methods in terms of CPU time and PSNR.}
	\label{fig:5}
\end{figure}

\begin{figure}[!tb]
	\centering
	\includegraphics[width=0.85\textwidth,height=6.0cm]{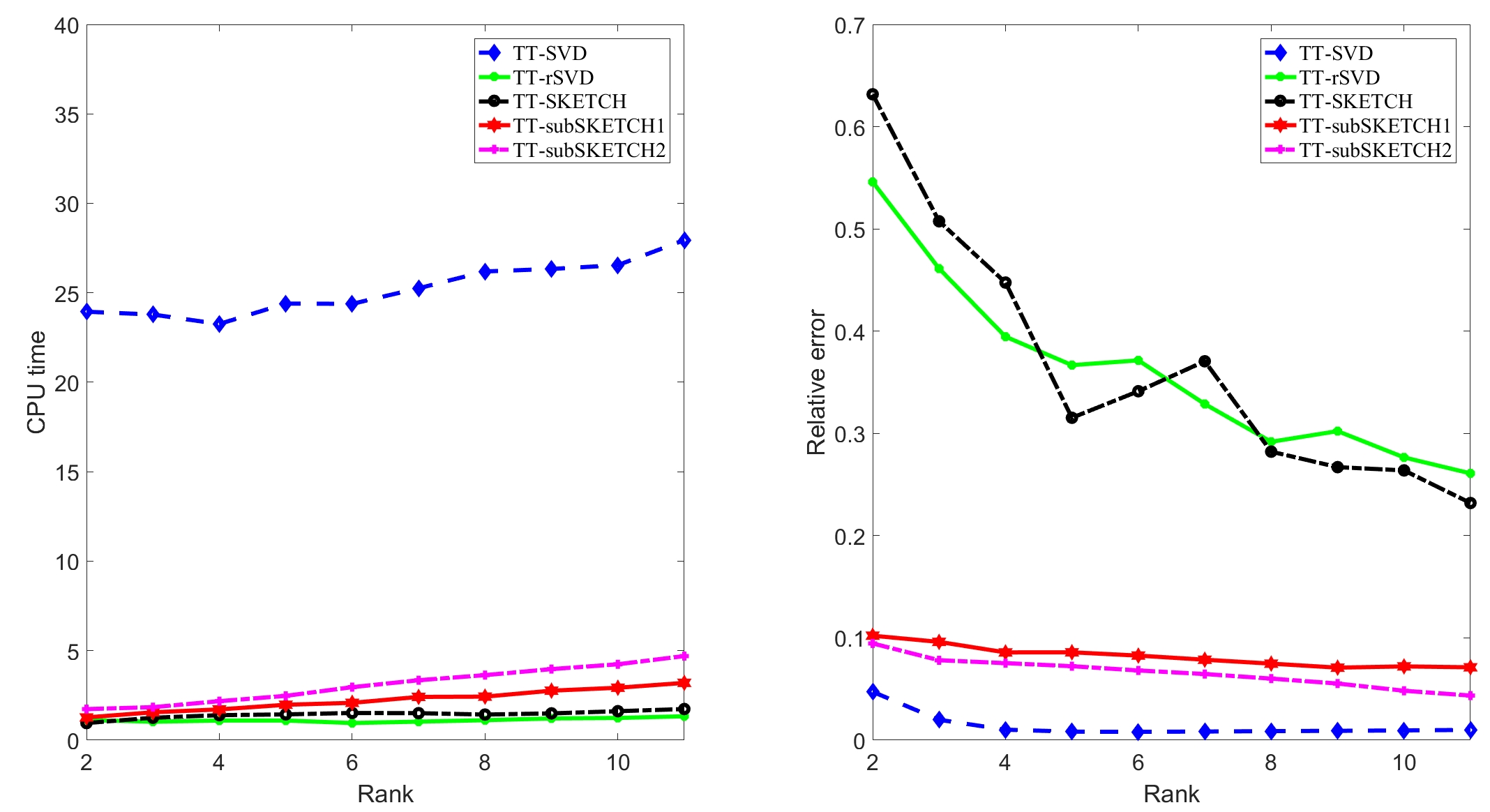}
	\caption{Results comparison on the noisy power function tensor with a size of $ 45\times 45\times 45\times 45\times 45 $ in terms of CPU time (left) and relative error (right) dependency on the target rank. The oversampling parameter p is fixed at 2, and the noise level is fixed at 2dB.}
	\label{fig:6}
\end{figure}

\begin{figure}[!tb]
	\centering
	\includegraphics[width=0.75\textwidth,height=6.5cm]{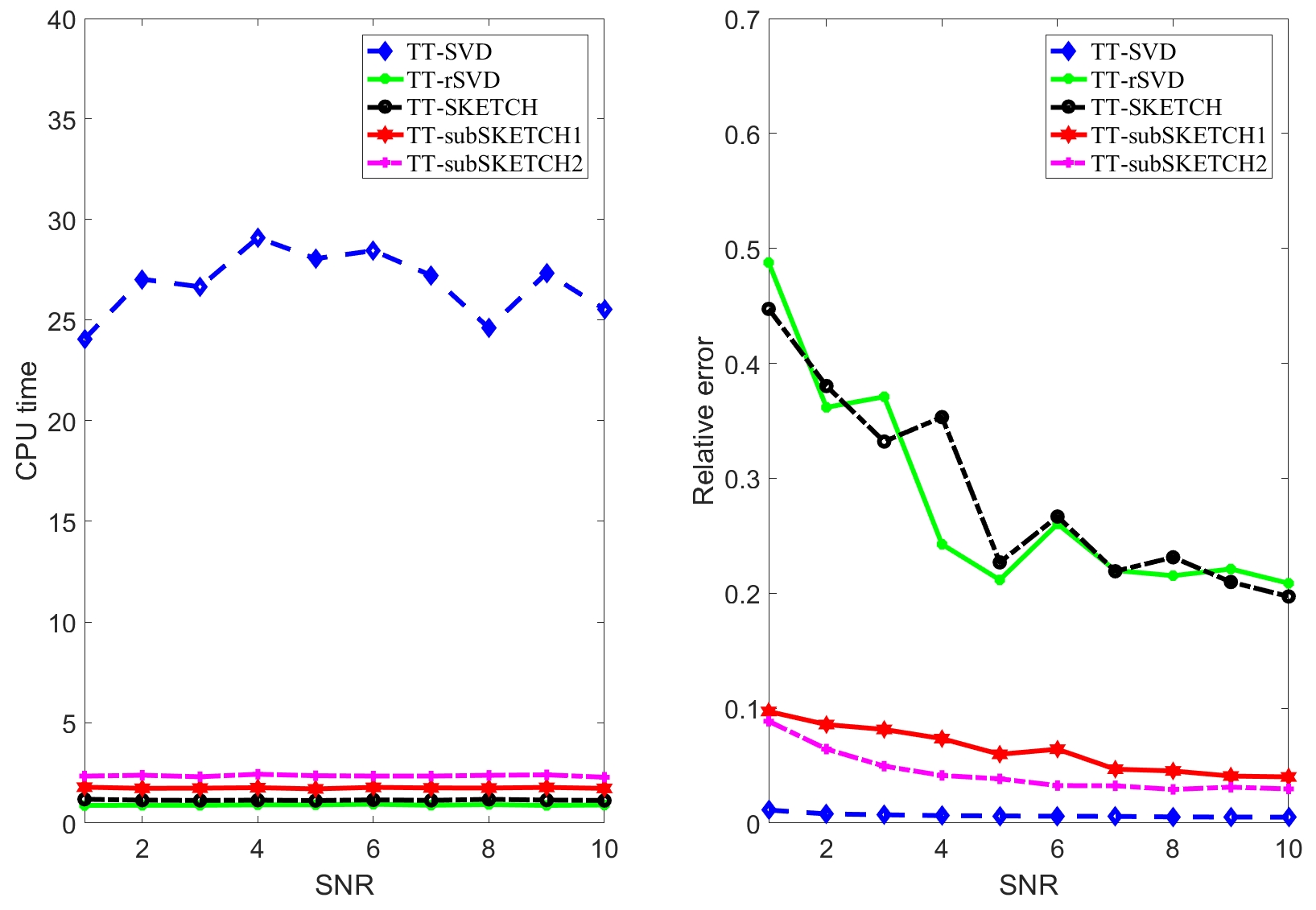}
	\caption{Results comparison on the noisy power function tensor with a size of $ 45\times 45\times 45\times 45\times 45 $ in terms of CPU time (left) and relative error (right) dependency on the SNR. The oversampling parameter p is fixed at 2, and the target rank is fixed at (5,5,5,5).}
	\label{fig:7}
\end{figure}

\subsection{Noisy data}
We finally consider the performance of different algorithms on synthetic and real-world data with noise, i.e., Gaussian white noise. The MATLAB built-in function \texttt{awgn}($\cdot$, SNR) is used to generate noise.

For the results on synthetic data, we keep the parameters the same as the settings in Section 4.1 and fix the noise level at 2dB. Figure \ref{fig:6} gives the results on the power function tensor of different algorithms in terms of the CPU time and relative error corresponding to different target ranks. Moreover, to further demonstrate the impact of noise on individual methods, we maintain the target rank constant and systematically change the noise level, see the results in Figure \ref{fig:7}.
Note that a higher SNR indicates lower noise levels, implying that with lower noise levels, the approximation error of different algorithms is lower. Among random algorithms, the results show that our TT-subSKETCH2 achieves the highest approximation quality; in particular, it requires much less time compared to TT-SVD.


\begin{figure}[!tb]
	\centering
	\includegraphics[width=0.9\textwidth,height=6.2cm]{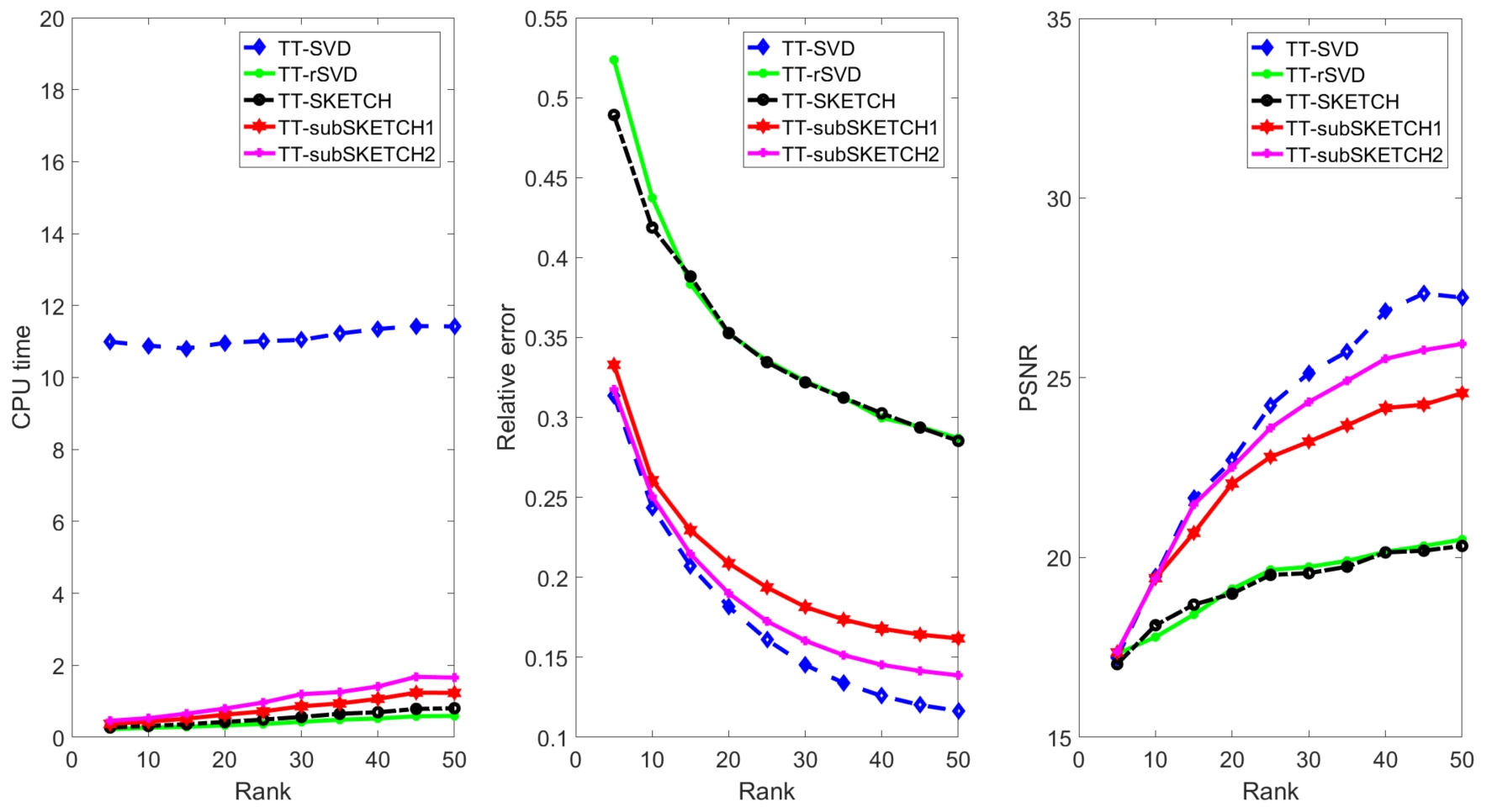}
	\caption{Results comparison on the noisy color video with a size of $ 360\times 640\times 3\times 100 $ in terms of CPU time (left), relative error (middle), and PSNR (right) dependency on the target rank. The oversampling parameter p is fixed at 5, and the noise level is fixed at 5dB.}
	\label{fig:8}
\end{figure}

\begin{figure}[!tb]
	\centering
    \begin{tabular}{ccc}
	\includegraphics[trim={{.01\linewidth} {1.30\linewidth} {1.91\linewidth} {.01\linewidth}}, clip, width=0.3\textwidth, height=0.18\textwidth]{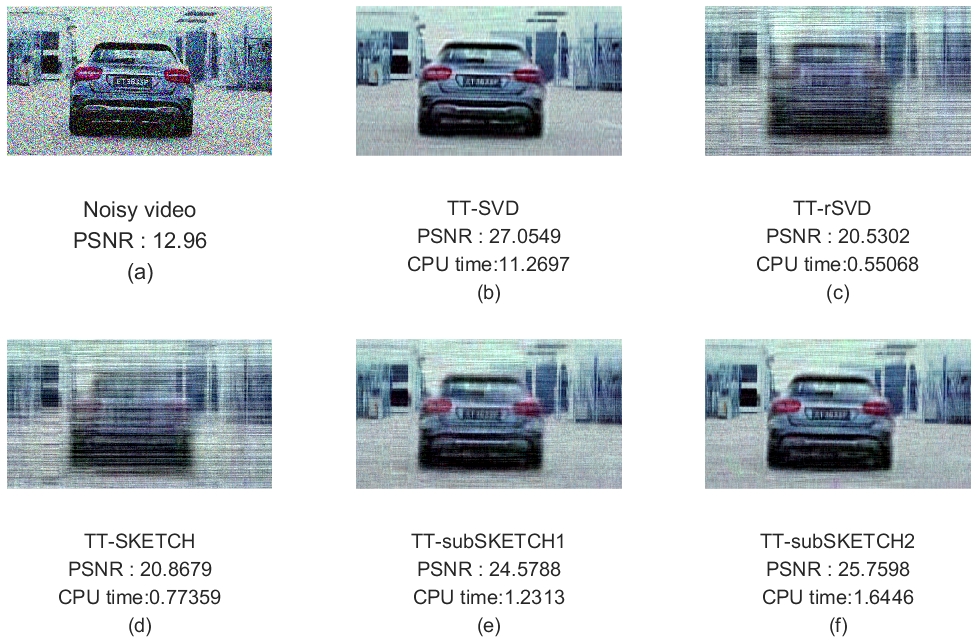} 
    \put(-95, -4){\tiny{Given original noisy frame}} 
    \put(-80, -11){\tiny{PSNR: 12.9600}} & 
    \includegraphics[trim={{0.97\linewidth} {1.30\linewidth} {0.96\linewidth} {.01\linewidth}}, clip, width=0.3\textwidth, height=0.18\textwidth]{Driving_noise.png}
    \put(-70, -4){\tiny{TT-SVD}}
    \put(-107, -11){\tiny{PSNR: 27.0549; CPU time: 11.27}} & 
    \includegraphics[trim={{1.91\linewidth} {1.30\linewidth} {0.01\linewidth} {.01\linewidth}}, clip, width=0.3\textwidth, height=0.18\textwidth]{Driving_noise.png}
    \put(-70, -4){\tiny{TT-rSVD}}
    \put(-105, -11){\tiny{PSNR: 20.5302; CPU time: 0.55}} \\
    \includegraphics[trim={{.01\linewidth} {.40\linewidth} {1.91\linewidth} {0.91\linewidth}}, clip, width=0.3\textwidth, height=0.18\textwidth]{Driving_noise.png} 
    \put(-73, -4){\tiny{TT-SKETCH}}
    \put(-103, -11){\tiny{PSNR: 20.8679; CPU time: 0.77}} & 
    \includegraphics[trim={{0.97\linewidth} {0.40\linewidth} {0.96\linewidth} {0.91\linewidth}}, clip, width=0.3\textwidth, height=0.18\textwidth]{Driving_noise.png} 
     \put(-81, -4){\tiny{TT-subSKETCH1}}
    \put(-106, -11){\tiny{PSNR: 24.5788; CPU time: 1.23}} & 
    \includegraphics[trim={{1.91\linewidth} {0.40\linewidth} {0.01\linewidth} {.91\linewidth}}, clip, width=0.3\textwidth, height=0.18\textwidth]{Driving_noise.png} 
     \put(-81, -4){\tiny{TT-subSKETCH2}}
    \put(-105, -11){\tiny{PSNR: 25.7598; CPU time: 1.64}}
      
    \end{tabular}
    
	\caption{Low-rank approximation of one frame of the color video clip with noise level fixed at 5dB by different methods in terms of CPU time and PSNR.}
	\label{fig:9}
\end{figure}


\begin{figure}[!tb]
	\centering
	\includegraphics[width=0.9\textwidth,height=6.0cm]{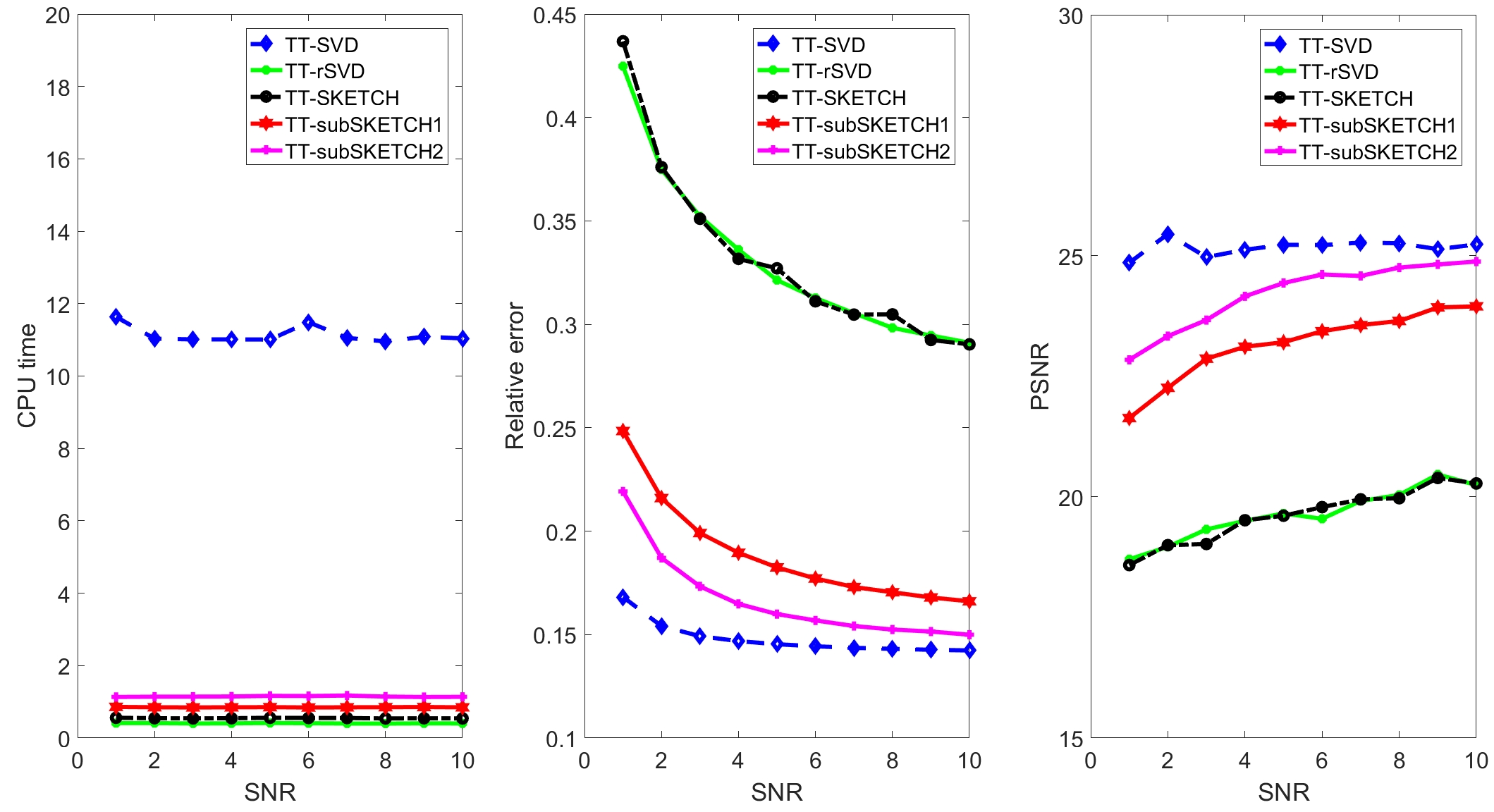}
	\caption{Results comparison on the noisy color video with a size of $ 360\times 640\times 3\times 100 $ in terms of CPU time (left), relative error (middle) and PSNR (right) dependency on the noise level. The oversampling parameter p is fixed at 5, and the target rank is fixed at (30,30).}
	\label{fig:10}
\end{figure}

For the results on real-world data, the color video data is used with parameters the same as that of Section 4.2. The quantitative and qualitative experimental results are shown in Figure \ref{fig:8} and Figure \ref{fig:9}, respectively, in terms of CPU time, relative error, and PSNR.
%
The results show that our TT-subSKETCH2 again achieves the approximation quality closest to TT-SVD in the presence of noise, demonstrating the superiority and robustness of the proposed algorithm. Figure \ref{fig:10} gives the results of different methods on the color video data with different level of noise under fixed rank. Consistent results are obtained, similar
to those on the synthetic data, further demonstrating the great performance of our proposed
algorithm.



\section{Conclusion}
\label{section:5}
In this paper, we proposed a two-sided sketching algorithm, named TT-subSKETCH, for TT approximation. 
Detailed error analysis and robustness analysis for the proposed algorithm TT-subSKETCH were conducted. Thorough experimental results on both synthetic and real-world data demonstrate that the proposed algorithm achieves superior accuracy and robustness compared to the related state-of-the-art random algorithms, while requiring significantly less time compared to deterministic algorithms.

\appendix
\section{Technical lemmas} 
Here we compile the essential technical lemmas that are pivotal in the proof of the proposed theorems.

\begin{lemma}
\label{Lem1}
(\cite{dong2023}, Lemma 4). Assume that the sketch size parameter satisfies $ l>k+1 $. Draw random test matrices $ \Omega\in\mathbb{R}^{n\times k} $ and $ \Phi\in\mathbb{R}^{l\times m} $ independently of the standard normal distribution. Then the rank-$k$ approximation $ \hat{A} $ obtained from Algorithm \ref{alg:sub} satisfies
			\begin{equation}
				\mathbb{E}\lVert A-\hat{A}\rVert_F^2\leq(1+f(k,l))\min\limits_{\varrho<k-1}(1+f(\varrho,k)\varpi_k^{4q})\cdot\tau_{\varrho+1}^2(A) \ ,
		\end{equation}
	where $ \varpi_k=\sigma_{k+1}/\sigma_k $ is the singular value gap, and $ f(s,t):=s/(t-s-1) $.
\end{lemma}


\begin{lemma}
    \label{Lem8}
(Presenting the explicit expression of $ X-Q^\top A $ in \cite{tropp2017}, Lemma A.4).
     Construct a matrix $ P\in\mathbb{R}^{n\times (n-k)} $ with orthonormal columns satisfying
	\begin{equation}
		PP^\top = I-QQ^\top .
	\end{equation}
Introduce the matrices
\begin{equation}
	\Psi_1:=\Psi P\in\mathbb{R}^{l\times (n-k)}, \ \  \Psi_2:=\Psi Q\in\mathbb{R}^{l\times k}.
\end{equation}
Assume that the matrix $\Psi_2$ has full column rank. Then
\begin{equation}
	X-Q^\top A=\Psi_2^{\dagger}\Psi_1(P^\top A) \ .
\end{equation}

\end{lemma}

\begin{lemma}
		\label{Lem2}
		(Pseudo-inverse properties, see e.g. \cite{golub2013}). The following properties hold for the pseudo-inverse.
  \begin{enumerate} 
  \item If $A$ has linearly independent columns, then $A^\dagger A=I$. If $B$ has linearly independent rows, then $BB^\dagger=I$. 
  \item $(AB)^\dagger=B^\dagger A^\dagger$ if $A$ has orthonormal columns or $B$ has orthonormal rows.
  \item $(AB)^\dagger=B^\dagger A^\dagger$ if $A$ has linearly independent columns.
  \item If $A$ has orthonormal columns or orthonormal rows, then $ A^\dagger=A^\top $.
  \end{enumerate}
\end{lemma}

\begin{lemma}
    \label{Lem3}
    (\cite{hansen2004}, Theorem 2.1). Let $P\in\mathbb{R}^{n\times n}$ be a projection matrix, i.e., $P^2=P$. We observe that $P$ is distinctively defined by the subspaces $V_1$ and $V_2$, i.e.,
    \begin{equation}
        V_1 := {\rm Im} P, \ \ 
        V_2 := {\rm Ker} P,
    \end{equation}
    and $\dim V_1 + \dim V_2 = n$. We denote by $V_1$ a matrix with orthonormal columns, whose column span is equal to the subspace $V_1$. Analogously, we denote by $V_2$ a matrix with orthonormal columns, whose column span is equal to the subspace $V_2$. It holds that
    \begin{equation}
        P=V_1(V_{2,\perp}^\top V_1)^{-1}V_{2,\perp}^\top \ .
    \end{equation}
\end{lemma}

\begin{lemma}
    \label{Lem4}
    (\cite{banks2023}, Proposition C.3). Let $A$ be the upper left $(n-r)\times(n-r)$ corner of a Haar unitary matrix $U\in\mathbb{R}^{n\times n}$. Then  $\forall\delta>0$,
    \begin{equation}
        \mathbb{P}\Big(\sigma_{\min}(A)\geq\frac{\sqrt{\delta}}{\sqrt{r(n-r)}}\Big) \geq 1-\delta \ .
    \end{equation}
\end{lemma}
\begin{lemma}
    \label{Lem5}
    (\cite{tao2010}, Theorem 1.1). Let $A$ be an $n\times n$ standard Gaussian random matrix. Then, $\forall \delta>0$,
    \begin{equation}
        \mathbb{P}\Big(\sigma_{\min}(A)\geq\frac{\delta}{\sqrt{n}}\Big) = \exp(-\delta^2/2-\delta) \ .
    \end{equation}
\end{lemma}
\begin{lemma}
    \label{Lem6}
    (\cite{vershynin2010}, Corollary 5.35). Let $A$ be an $m\times n$ matrix whose entries are independent standard normal random variables. Then $\forall \delta\in(0,1]$, with probability at least $1-\delta$, we have
    \begin{align}
        \sigma_{\min}(A) & \geq \sqrt{m}-\sqrt{n}-\sqrt{2\log(2/\delta)} \ , \\
        \sigma_{\max}(A) & \leq \sqrt{m}+\sqrt{n}+\sqrt{2\log(2/\delta)} \ .
    \end{align}
\end{lemma}

\section*{Acknowledgements}
	This work was supported in part by National Natural Science Foundation of China (No. 12071104).

\section*{Ethics Declarations}
 We declare that we have no commercial or associative interest that represents a conflict of interest in connection with the work submitted.

 \section*{Data Availability}
All underlying data sets can be made available upon request or are publicly available. MATLAB codes used in this paper are available upon request to the authors. 

\bibliographystyle{siamplain}
\bibliography{references}
\end{document}

%% file: ex_shared.tex
\usepackage{lipsum}
\usepackage{amsfonts}
\usepackage{graphicx}
\usepackage{epstopdf}
\usepackage{algorithmic}
\usepackage{algorithm}
\usepackage{amssymb}
\ifpdf
  \DeclareGraphicsExtensions{.eps,.pdf,.png,.jpg}
\else
  \DeclareGraphicsExtensions{.eps}
\fi

\newsiamremark{remark}{Remark}
\newsiamremark{hypothesis}{Hypothesis}
\crefname{hypothesis}{Hypothesis}{Hypotheses}
\newsiamthm{claim}{Claim}
\newtheorem{Def}{Definition}

\headers{Low-rank Tensor-Train Approximation}{Gaohang Yu, Yihao Pan, Ailun Jian, Xiaohao Cai}

\title{A Two-Sided Sketching Algorithm for Low-rank Tensor-Train Approximation\thanks{Submitted to the editors DATE.
\funding{This work was funded by the }}}

\author{Gaohang Yu\thanks{School of Science, Zhejiang University of Science and Technology, Hangzhou, China. (\email{maghyu@163.com}).}
\and Yihao Pan \thanks{School of Science, Hangzhou Dianzi University, Hangzhou, China. (\email{294002507@qq.com})}
\and Ailun Jian \thanks{School of Science, Hangzhou Dianzi University, Hangzhou, China. (\email{allenjian805@163.com})}
\and xiaohao Cai\thanks{School of Electronics and Computer Science, University of Southampton, Southampton, UK. (\email{x.cai@soton.ac.uk}).}}

\usepackage{amsopn}
